\documentclass[doublespacing]{utdthesis}
\usepackage{amsmath,amssymb,amsthm}
\usepackage{graphicx}
\usepackage{url}

\makeatletter
\renewcommand{\@biblabel}[1]{[#1]\hfill}
\makeatother

\usepackage{hyperref}
\hypersetup{colorlinks, linkcolor=reference, citecolor=citation, urlcolor=e-mail}

\PassOptionsToPackage{usenames,dvipsnames}{xcolor}

\usepackage{enumerate}
\usepackage{tikz}
\usetikzlibrary{cd}
\usepackage{quiver}
\usepackage{soul}
\usepackage{marginnote}
\usepackage{bm}
\usepackage{comment}

\DeclareFontShape{OMX}{cmex}{m}{b}{<-> cmexb10}{}
\DeclareSymbolFont{boldlargesymbols}{OMX}{cmex}{m}{b}

\let\div\relax
\DeclareMathOperator{\div}{div}
\DeclareMathOperator{\Div}{Div}
\DeclareMathOperator{\Id}{Id}
\DeclareMathOperator{\Aut}{Aut}
\DeclareMathOperator{\Span}{Span}
\DeclareMathOperator{\ores}{Ores}
\DeclareMathOperator{\pano}{PanOres}
\DeclareMathOperator{\res}{res}
\DeclareMathOperator{\sing}{sing}
\DeclareMathOperator{\pdisp}{pdisp}
\DeclareMathOperator{\wpdisp}{wpdisp}

\newcommand{\fieldC}{\mathcal{K}_{\Lambda}}

\newcommand{\fieldAlg}{\mathcal{K}_{1}}

\newcommand{\allOmegas}{\mathbb{C}/(\Lambda\oplus\mathbb{Z}\shiftnumber)}
\newcommand{\shift}{\tau}

\newcommand{\shiftnumber}{s}
\newcommand{\shiftalg}{S}
\newcommand{\shifttorus}{\shift^*}

\newcommand{\goRight}{
\begin{tikzpicture}[x=10pt,y=10pt,baseline=-0.7ex]
\draw[->] (0,0)--(1,0);
\filldraw[black] (0,0) circle (1pt);
\filldraw[black] (1,0) circle (1pt);
\filldraw[black] (-1,0) circle (1pt);
\filldraw[black] (0,1) circle (1pt);
\filldraw[black] (1,1) circle (1pt);
\filldraw[black] (-1,1) circle (1pt);
\filldraw[black] (0,-1) circle (1pt);
\filldraw[black] (1,-1) circle (1pt);
\filldraw[black] (-1,-1) circle (1pt);
\end{tikzpicture}
}

\newcommand{\goRightUp}{
\begin{tikzpicture}[x=10pt,y=10pt,baseline=-0.7ex]
\draw[->] (0,0)--(1,1);
\filldraw[black] (0,0) circle (1pt);
\filldraw[black] (1,0) circle (1pt);
\filldraw[black] (-1,0) circle (1pt);
\filldraw[black] (0,1) circle (1pt);
\filldraw[black] (1,1) circle (1pt);
\filldraw[black] (-1,1) circle (1pt);
\filldraw[black] (0,-1) circle (1pt);
\filldraw[black] (1,-1) circle (1pt);
\filldraw[black] (-1,-1) circle (1pt);
\end{tikzpicture}
}

\newcommand{\goUp}{
\begin{tikzpicture}[x=10pt,y=10pt,baseline=-0.7ex]
\draw[->] (0,0)--(0,1);
\filldraw[black] (0,0) circle (1pt);
\filldraw[black] (1,0) circle (1pt);
\filldraw[black] (-1,0) circle (1pt);
\filldraw[black] (0,1) circle (1pt);
\filldraw[black] (1,1) circle (1pt);
\filldraw[black] (-1,1) circle (1pt);
\filldraw[black] (0,-1) circle (1pt);
\filldraw[black] (1,-1) circle (1pt);
\filldraw[black] (-1,-1) circle (1pt);
\end{tikzpicture}
}

\newcommand{\goUpLeft}{
\begin{tikzpicture}[x=10pt,y=10pt,baseline=-0.7ex]
\draw[->] (0,0)--(-1,1);
\filldraw[black] (0,0) circle (1pt);
\filldraw[black] (1,0) circle (1pt);
\filldraw[black] (-1,0) circle (1pt);
\filldraw[black] (0,1) circle (1pt);
\filldraw[black] (1,1) circle (1pt);
\filldraw[black] (-1,1) circle (1pt);
\filldraw[black] (0,-1) circle (1pt);
\filldraw[black] (1,-1) circle (1pt);
\filldraw[black] (-1,-1) circle (1pt);
\end{tikzpicture}
}

\newcommand{\goLeft}{
\begin{tikzpicture}[x=10pt,y=10pt,baseline=-0.7ex]
\draw[->] (0,0)--(-1,0);
\filldraw[black] (0,0) circle (1pt);
\filldraw[black] (1,0) circle (1pt);
\filldraw[black] (-1,0) circle (1pt);
\filldraw[black] (0,1) circle (1pt);
\filldraw[black] (1,1) circle (1pt);
\filldraw[black] (-1,1) circle (1pt);
\filldraw[black] (0,-1) circle (1pt);
\filldraw[black] (1,-1) circle (1pt);
\filldraw[black] (-1,-1) circle (1pt);
\end{tikzpicture}
}

\newcommand{\goLeftDown}{
\begin{tikzpicture}[x=10pt,y=10pt,baseline=-0.7ex]
\draw[->] (0,0)--(-1,-1);
\filldraw[black] (0,0) circle (1pt);
\filldraw[black] (1,0) circle (1pt);
\filldraw[black] (-1,0) circle (1pt);
\filldraw[black] (0,1) circle (1pt);
\filldraw[black] (1,1) circle (1pt);
\filldraw[black] (-1,1) circle (1pt);
\filldraw[black] (0,-1) circle (1pt);
\filldraw[black] (1,-1) circle (1pt);
\filldraw[black] (-1,-1) circle (1pt);
\end{tikzpicture}
}

\newcommand{\goDown}{
\begin{tikzpicture}[x=10pt,y=10pt,baseline=-0.7ex]
\draw[->] (0,0)--(0,-1);
\filldraw[black] (0,0) circle (1pt);
\filldraw[black] (1,0) circle (1pt);
\filldraw[black] (-1,0) circle (1pt);
\filldraw[black] (0,1) circle (1pt);
\filldraw[black] (1,1) circle (1pt);
\filldraw[black] (-1,1) circle (1pt);
\filldraw[black] (0,-1) circle (1pt);
\filldraw[black] (1,-1) circle (1pt);
\filldraw[black] (-1,-1) circle (1pt);
\end{tikzpicture}
}

\newcommand{\goDownRight}{
\begin{tikzpicture}[x=10pt,y=10pt,baseline=-0.7ex]
\draw[->] (0,0)--(1,-1);
\filldraw[black] (0,0) circle (1pt);
\filldraw[black] (1,0) circle (1pt);
\filldraw[black] (-1,0) circle (1pt);
\filldraw[black] (0,1) circle (1pt);
\filldraw[black] (1,1) circle (1pt);
\filldraw[black] (-1,1) circle (1pt);
\filldraw[black] (0,-1) circle (1pt);
\filldraw[black] (1,-1) circle (1pt);
\filldraw[black] (-1,-1) circle (1pt);
\end{tikzpicture}
}

\RequirePackage{color}
\definecolor{e-mail}{rgb}{0,.40,.80}
\definecolor{reference}{rgb}{.20,.60,.22}
\definecolor{citation}{rgb}{0,.40,.80}
\usepackage{ragged2e}
\usepackage{enumitem}

\usepackage{makecell}

\usepackage{thmtools}
\usepackage[framemethod=TikZ]{mdframed}
\mdfsetup{skipabove=1em,skipbelow=0em}
\theoremstyle{definition}
\declaretheoremstyle[
    headfont=\bfseries\color{black}, bodyfont=\normalfont,
    mdframed={
        linewidth=2pt,
        rightline=false, topline=false, bottomline=false,
        linecolor=black, backgroundcolor=white,
    }
]{contentBox}

\declaretheorem[style=contentBox, name=Definition, numberwithin=section]{definition}
\declaretheorem[style=contentBox, name=Theorem, sibling=definition]{theorem}
\declaretheorem[style=contentBox, name=Lemma, sibling=definition]{lemma}
\declaretheorem[style=contentBox, name=Proposition, sibling=definition]{proposition}

\declaretheorem[style=contentBox, name=Problem, numbered=no]{problem}
\declaretheorem[style=contentBox, name=Question, numbered=no]{question}
\declaretheorem[style=contentBox, name=Corollary, sibling=definition]{corollary}
\declaretheorem[style=contentBox, name=Example, sibling=definition]{example}
\declaretheorem[style=contentBox, name=Remark, sibling=definition]{remark}

\newcommand\defemph[1]{\textbf{\emph{#1}}}

\author{Matthew William Babbitt}
\title{Summability of Elliptic Functions Via Residues}
\thesistype{Dissertation}
\degreefull{Doctor of Philosophy}
\degreeabbr{PhD}
\subject{Mathematics}
\graduationmonth{May}
\graduationyear{2025}
\prevdegrees{BS}

\committeemember*{Carlos Arreche}
\committeemember{Mieczyslaw Dabkowski}
\committeemember{Vladimir Dragovic}
\committeemember{Nathan Williams}
\committeemember{John Zweck}

\usepackage{subfiles}

\begin{document}

\frontmatter

\maketitle

\copyrightpage{2025}

\begin{dedication}
This thesis is dedicated \\
to my parents William and Nancy, \\
and to my brother Joshua, \\
for putting up with my math antics.
\end{dedication}

\begin{acks}{February 2025}

First and foremost my eternal gratitude goes to my advisor, Carlos Arreche, for the unwavering support he has given me in this endeavor. In our invaluable meetings he not only guided me through independent study, research, and writing, but he also helped me navigate math academia and showed endless kindness through my tribulations. He has shown me that mentorship is an art form and I hope to be like him when I grow up.

I am indebted to many of the faculty at the University of Texas at Dallas who have also guided me over the past several years. I am very grateful for Vladimir Dragovi\'{c} and Viswanath Ramakrishna for giving me opportunities to grow into teaching and mentorship through REUs and summer programs. Coaching the university Putnam team with Nathan Williams gave me the opportunity to show undergraduates the breadth of mathematics and I will remember those moments forever. John Zweck helped get my math career reoriented in a single session of office hours in my first semester, and since then I have been working hard to pay this tremendous gift forward.

There were many more conversations that I had with professors that I will remember and cherish, including Mohammad Akbar, Maxim Arnold, the late Zalman Balanov, Ronan Conlon, Baris Coskunuzer, and Mieczyslaw Dabkowski. I am also grateful to the administrators Kisa Sarachchandra and James Smith who both worked endlessly with me to help get every form and key to where it needed to go.

I have been blessed to have met Sri Rama Chandra Kushtagi within the last 3 years of my doctoral studies. Not only has Chandra has taught me how to read and write, but he has also shown me that there is much culture to be found within the mathematical community. In short, he has shown me that math is very much a human endeavor. I wish him the best of luck with his thesis defense and his book collection.

I owe my life to Patrick and Ema, and to their parents Harry, Linda, Marian, and Mihaela. They kept me when I could not secure housing during the height of the pandemic. I will treasure our myriad conversations, their enormously generous transportation to and from campus, and Mihaela's allergy test that helped me find how many foods I am now allergic to. I would not have finished my studies if not for them.

I have been fortunate to have many friends that have taught me, supported me, and kept me sane. There are too many to list, but I can certainly try. I have arranged them into groups, perhaps only recognizable to their members:

FDR, Julian, Keaton, Will, Willie,
Olivia, Adam, Alex, Andrew, Jareth, Jaz, Izzy,
David, Katie, Zach, NES and its Sheriffs,
Bhanu, Brighton, Safi, Mahesh, Josean, Russell,
Alavi, Colton, Grayson, Kevin,
Alina, Titu, David, Hiram,
Alexandra, Mukkai, Catherine, Mary,
Chris, Eshan, and George.

And finally, my studies would not have started in the first place had it not been for my family. My parents William and Nancy shifted their entire lives to help me pursue math through career changes, hundreds of hours of car rides to math events, and their unconditional support even when I fail. My brother Joshua has always been in my corner cheering me on. You have done more for me than thanks can convey.

\end{acks}

\begin{abstract}
Summability has been a central object of study in difference algebra over the past half-century. It serves as a cornerstone of algebraic methods to study linear recurrences over various fields of coefficients and with respect to various kinds of difference operators. Recently, Dreyfus, Hardouin, Roques, and Singer introduced a notion of elliptic orbital residues, which altogether serve as a partial obstruction to summability for elliptic functions with respect to the shift by a non-torsion point over an elliptic curve. We explain how to refine this into a complete obstruction, which promises to be useful in applications of difference equations over elliptic curves, such as elliptic hypergeometric functions and the combinatorics of walks in the quarter plane.
\end{abstract}

\tableofcontents

\mainmatter
\renewcommand\arraystretch{0.6}

\chapter{Introduction}
\label{c:intro}

Given a field $\mathcal{K}$ and an automorphism $\tau\in\Aut(\mathcal{K})$, so-called \defemph{first-order linear difference equations} of the form
\begin{equation}
\tau(y)-y=b\label{eq:diffeqord1}
\end{equation}
for $b\in\mathcal{K}$ and formal variable $y$, are of interest in the field of difference algebra. Should there exist a solution $y=m$ to the above equation~\eqref{eq:diffeqord1} where $m\in \mathcal{K}$, we say that $b$ is \defemph{summable}. This leads naturally to the question of which elements of $\mathcal{K}$ are summable: the dissertation before you lays a theoretical framework for answering this question in the case that $\mathcal{K}$ is the field of meromorphic functions over an elliptic curve $E$, and $\tau$ is a shift automorphism on $\mathcal{K}$ by a non-torsion point of $E$.

This work is positioned inside a larger program of calculating \defemph{differential Galois groups of linear difference equations} over elliptic curves. Deciding the summability of an elliptic function is a key subproblem. This serves in particular as a part of the computation of such differential Galois groups for second-order linear difference equations over elliptic curves. The goal is to perform this computation in an algorithmic manner whereas current methods are either ad-hoc or partial in nature. While those methods serve their applications well, a systematic algorithm for the computation can facilitate further, more demanding results.

\section{Difference Equations over Elliptic Curves}

Elliptic curves first arose in connection with the problem of calculating arc lengths of ellipses. This quickly evolved into the study of doubly-periodic functions on the complex plane $\mathbb{C}$. Since then, and particularly in more recent decades, they have seen more number theoretic applications, such as in the proof of Fermat's Last Theorem and in the construction of cryptographic schemes. In the past decade there have been some additional and surprising applications in more disparate fields: in the combinatorial study of analyzing walks confined to the first quadrant of the plane, and in an application towards mathematical physics via analyzing elliptic hypergeometric functions.

Elliptic curves can be defined in several different ways: for example, see \cite{Silverman2009,silvermanAdvancedTopicsArithmetic1994}. Most familiar is a complex-analytic construction of the quotient space of $\mathbb{C}$ modulo a $\mathbb{Z}$-lattice $\Lambda$. Less familiar and more general is the multiplicative version $K^{\times} / q^{\mathbb{Z}}$, often called a Tate curve, where $K$ is a complete normed field and $|q|<1$, covered in \cite[Ch.~V]{silvermanAdvancedTopicsArithmetic1994}. In algebraic geometry, an elliptic curve can be defined over any field $K$ as a curve of genus one with a distinguished $K$-rational point. In all contexts elliptic curves have the prominent feature of a \emph{group law} $\oplus$ on the points of the curve, and the difference equations of interest are defined in terms of this group operation.

Many important phenomena concerning an elliptic curve $E$ can be seen at the level of the field $\mathcal{K}$ of meromorphic functions on $E$. The group law on $E$ can be used to endow $\mathcal{K}$ with the structure of a \defemph{difference field} as follows. Given a point\footnote{This point is usually assumed for technical reasons to be \defemph{non-torsion}, i.e. have infinite order under the group law.} $s\in E$, the shift $z\mapsto z\oplus s$ on $E$ induces a shift automorphism
\begin{align*}
    \tau&:\mathcal{K}\to \mathcal{K}\\
    \tau&:f(z)\mapsto f(z\oplus s).
\end{align*}
A \defemph{linear difference equation} over the elliptic curve $E$ is an equation of the form
\begin{equation}
\sum\limits_{i=0}^{n} a_i(z)\cdot \tau^i(y) = \sum\limits_{i=0}^{n} a_i(z)\cdot f(z\oplus is) = b(z),\label{eq:diffEq}
\end{equation}
where $a_i(z)\in \mathcal{K}$ for all $1\leq i\leq n$, where $a_0\cdot a_n\neq 0$, and $y = f(z)$ is an unknown ``function'' (or perhaps a more general entity like a formal power series). If $b(z)=0$ then the equation is said to be homogeneous.

\section{Recent Applications}
\label{s:ellipticApplications}

There are two recent applications of linear difference equations over elliptic curves that are highly interesting because they lie in disparate fields of study. We briefly touch on two of these applications here to demonstrate their versatility.

Suppose we wish to count the number of walks from $(0,0)$ to $(i,j)$ using precisely $k$ steps, where each step is taken from a prescribed choice $\mathcal{D}\subseteq \{(0,\pm 1)\}^2 \, \backslash \,  \{(0,0)\}$ of permissible steps, i.e.,
\[\mathcal{D}\subseteq \Bigg\{ \ \goRight \ , \quad \goRightUp \ , \quad \goUp \ , \quad \goUpLeft \ , \quad \goLeft \ , \quad \goLeftDown \ , \quad \goDown \ , \quad \goDownRight \ \Bigg\}\]
and where each walk is forced to always stay within the first quadrant $\min\{x,y\}\geq 0$. Accounting for symmetry and discounting trivial choices of $\mathcal{D}$, there are only $79$ ``interesting'' sets $\mathcal{D}$ to be studied. Given a prescribed set $\mathcal{D}$ of admissible moves, one can then construct the generating series
\[Q_{\mathcal{D}}(x,y,t) = \sum\limits_{i,j,n\geq 0} w_{\mathcal{D}}(i,j,n) x^iy^jt^n\]
counting the number $w_{\mathcal{D}}(i,j,n)$ of walks following $\mathcal{D}$ staying in the first quadrant and ending at the point $(x,y)=(i,j)$ in exactly $n$ moves. It can then be shown that $Q_{\mathcal{D}}(x,y,t)$ satisfies a functional equation which, in $51$ of the interesting choices for $\mathcal{D}$, can be interpreted as a first-order linear difference equation of the form $\tau(f) - f=b $ for some meromorphic function $b$ on a particular elliptic curve $E$. The authors of \cite{Dreyfus2018} systematically study $Q_{\mathcal{D}}$ for all $51$ choices of $\mathcal{D}$ at once, using the data of the equation $\tau(f)-f=b$ alone. They were able to \emph{systematically} show which of the $Q_{\mathcal{D}}$ are \defemph{hypertranscendental}: i.e. which satisfy that it and all its derivatives are algebraically independent. \\

A second application of linear difference equations over elliptic curves is to elliptic hypergeometric functions. These objects were studied in \cite{Spiridonov2008} as a generalization of classical Euler--Gauss hypergeometric functions $\phantom{.}_2F_1(a,b,c;z)$. For $p,q\in\mathbb{C}^*$ such that $p^\mathbb{Z}\cap q^\mathbb{Z}=\{1\}$ and  $|p|,|q|<1$, and $\underline{t}=(t_1,\dots,t_8)\in(\mathbb{C}^*)^8$ satisfying $\prod_{j=1}^8t_j=p^2q^2$, the \defemph{elliptic hypergeometric function} is
%
\[ V(\underline{t}\ ;p,q):=\int_{S^1}\frac{\prod_{j=1}^8\Gamma(t_jz;p,q)\Gamma(t_jz^{-1};p,q)}{\Gamma(z^2;p,q)\Gamma(z^{-2};p,q)}\cdot \frac{dz}{z}; \]
 where the \defemph{elliptic Gamma functions} appearing in the integrand are \begin{gather*}\Gamma(z;p,q):=(pq/z;p,q)_\infty/(z;p,q)_\infty; \quad \text{where} \quad (z;p,q)_\infty:=\prod_{j,k\geq 0}(1-p^jq^kz).\end{gather*}

It is shown in \cite{Spiridonov2020} that, after a certain change of variables, $V(\underline{t};p,q)=f_{\varepsilon}(z)$ satisfies a difference equation over the (Tate) elliptic curve $E=\mathbb{C}^*/q^\mathbb{Z}$ with respect to the (multiplicative) shift automorphism $\tau(f(z))=f(pz)$, whose coefficients are elements of $\mathcal{K}$ that depend on some parameters $\underline{\varepsilon}=(\varepsilon_1,\dots,\varepsilon_8)\in(\mathbb{C}^*)^8$. Recently it was shown in \cite{Arreche2021} that, for ``most'' choices of parameters $\underline{\varepsilon}$, the functions $f_{\underline{\varepsilon}}(z)$ are hypertranscendental over $\mathcal{K}$. \\

\subsection{Commonalities: Galois Theory}

The key to both of the above results is differential Galois theory for difference equations. This theory, developed in \cite{Hardouin2008}, associates with equation~\eqref{eq:diffEq} a geometric object called the differential Galois group which encodes in its algebraic structure the differential properties of the solutions to equation~\eqref{eq:diffEq}. In particular, if the Galois group is ``sufficiently large'' then the Galois correspondence implies that there are ``few'' differential equations satisfied by the solutions to equation~\eqref{eq:diffEq}. Indeed, the work of \cite{Dreyfus2018} and \cite{Arreche2021} in the applications above show that, for the specific difference equations over elliptic curves relevant to their respective applications, the associated differential Galois groups are so large that the solutions must be hypertranscendental.

Remarkably, these authors were successful in deducing the largeness of these differential Galois groups in the absence of a complete algorithm to compute them. However, such methods are inherently unable to yield \emph{positive} results: should the solutions of some difference equation over an elliptic curve actually satisfy some nontrivial differential equations, there are no systematic methods to discover them. This dissertation is part of a larger program to enable this discovery.

A future aim of the author is to develop a suite of theoretical and algorithmic results to compute differential Galois groups associated to second-order linear difference equations over elliptic curves. This has been done in the contexts of linear difference equations over rational functions for various difference automorphisms: the shift automorphism in \cite{Arreche2016}, and the $q$-difference automorphism in \cite{Arreche2022}. A common subproblem in both algorithms is determining the summability of a rational function with those automorphisms. The recent work of \cite{Chen2012} develops new technologies called \defemph{discrete residues} and \defemph{$q$-discrete residues}, numbers that can be extracted from a rational function, that detect whether that rational function is summable. The goal of this thesis is to extend this technique to the elliptic case.

\begin{problem}
Given a meromorphic function over an elliptic curve, how can we extract constants from the function that let us decide whether this function is summable? In other words, can we find constants that form an \emph{obstruction} to summability?
\end{problem}

\section{Dissertation Overview}

Here lists brief descriptions of the chapters of this dissertation. \\

Chapter~\ref{c:basicResidues} provides background for the larger program of differential Galois theory of linear difference equations, as well as previous results on summability. It begins with a brief history of the Galois theories of differential and difference equations. It is followed by a discussion of the discrete and $q$-discrete residues from \cite{Chen2012}, with small computational examples for motivation. Both the shift and $q$-difference cases are discussed here.

There are difficulties in the elliptic case that are absent from these two cases. These complications cannot be formulated until later, so in the subsequent chapters we will often contrast with ideas and results from Chapter~\ref{c:basicResidues}. \\

Chapter~\ref{c:analytic} analyzes summability in the analytic setting of doubly-periodic meromorphic functions on $\mathbb{C}$. This will be the most lightweight setting: the problem can be solved entirely with concepts from complex analysis. This chapter will serve to highlight technical difficulties in a more elementary setting, and to motivate the approaches of Chapter~\ref{c:algebraic}.

The methods in this chapter draw from the works of \cite{Dreyfus2018,Hardouin2021}, which extend the discrete and $q$-discrete residues of \cite{Chen2012} to the elliptic case with \defemph{orbital residues}. We reinterpret these objects in the analytic setting, and introduce additional objects called \defemph{panorbital residues} that are necessary to complete the obstruction for summability in the analytic setting. \\

Chapter~\ref{c:algebraic} analyzes summability in the setting of algebraic curves of genus 1. The foundational works \cite{Dreyfus2018,Hardouin2021} on orbital residues reside in this context precisely, so we will delay the discussion of the preliminary results until this chapter. We will also draw parallels between this context and Chapter~\ref{c:analytic} to show how their results can be detected already from the analytic setting. \\

Chapter~\ref{c:future} concludes the thesis by laying down some preliminary results that contextualize the results of the thesis within the broader study of differential Galois theory of linear difference equations over elliptic curves.

\chapter{Background on Difference Equations}
\label{c:basicResidues}

The study of differential Galois theory of linear difference equations is one in a long line of Galois theories that extend past the classical Galois theory of polynomial equations. The initial study is of the Galois theory of differential equations, also named \emph{Picard-Vessiot Theory}. The development of these theories is not only interesting, but it informs the approaches we will see in the coming chapters.

\section{A History of Picard-Vessiot Theory}

In order to study equations more involved than polynomial, we require algebras with more equipped machinery than a field and its operations. These constructs have their own subfields of algebra:

\begin{itemize}
    \item \defemph{Differential algebra}, which lets us analyze a field $K_1$ equipped with a \defemph{derivation} $\delta:K_1\to K_1$ satisfying the addition and Leibniz rules
    \[\delta(a+b) = \delta(a)+\delta(b),\qquad \delta(ab) = \delta(a)b+a\delta(b)\]
    for all $a,b\in K_1$. The data $(K_1,\delta)$ is referred to as a \defemph{differential field}.
    \item \defemph{Difference algebra}, which lets us analyze a field $K_2$ equipped with an automorphism $\tau$. The data $(K_2,\tau)$ is referred to as a \defemph{difference field}.
\end{itemize}
The study of these structures lets us systematically from an algebraic viewpoint analyze \defemph{linear differential equations} and \defemph{linear difference equations}, respectively of the forms
\begin{align*}
    \delta^n(y)+a_{n-1}\delta^{n-1}(y) + \cdots + a_1\delta(y)+a_0y=0,\\
    \tau^n(y)+b_{n-1}\tau^{n-1}(y)+\cdots + b_1\tau(y)+b_0y=0,
\end{align*}
for $a_i\in K_1$, $b_j\in K_2$, and $n\in\mathbb{N}$. Algebraic developments in the 19th and 20th century led to the birth and actualization of a Galois theory for each of these two types of equations, both of which were endowed with the name \defemph{Picard-Vessiot theory}. These fields of study are more commonly referred to as the Galois Theories of Linear Differential and Difference equations.

The usual study of Galois theory examines polynomial equations: to a polynomial with coefficients in a field $K$ is associated a splitting field $L$ containing its roots, and Galois theory studies the $K$-automorphisms of $L$. In 1948, Ellis R. Kolchin developed in \cite{Kolchin1948} such a theory for linear differential equations that uses a \defemph{Picard-Vessiot extension} in the place of a splitting field. The Galois theory of difference equations, however, requires the use of rings instead of fields, and was only formulated in 1997 by Marius van der Put and Michael F. Singer in \cite{vdPSingSigma}. There is little room in the thesis for a comprehensive treatment, but the interested reader can find a history of the Galois theory of linear differential equations in \cite[Chapter~VIII]{borelEssaysHistoryLie2001}, and the preface of \cite{vdPSingSigma} gives a precise timeline of the development of the Galois theory of linear difference equations.

\begin{remark}
    Borel's essay \cite{borelEssaysHistoryLie2001} on the Galois theory of linear differential equations is quite illuminating as it contextualizes its development alongside that of Lie theory and the theory of linear algebraic groups. It also explains the contributions of Picard and Vessiot to their namesake theory, whereas the name and the primary development of the theory are due to Kolchin!
\end{remark}

The Galois theories above are each focused on finding and analyzing a Galois group. In the case of the Galois theory of linear differential equations, one computes a \defemph{differential Galois group} of their linear differential equation, and this encodes the ``differential symmetries'' satisfied by the solutions to such an equation. An analogous viewpoint holds for the Galois theory of linear difference equations.

A natural question then is how one can find the \emph{differential} symmetries of a linear \emph{difference} equation. Hardouin and Singer established this new Galois theory in 2008 in \cite{Hardouin2008}, naming the paper and the theory the \emph{differential Galois theory of linear difference equations}. This dissertation ties directly into problems in this particular Galois theory.

\subsection{Second-Order Linear Difference Equations}

Each time a Galois theory of linear difference equations is formulated, mathematicians begin by applying it to second-order linear difference equations: those of the form
\begin{equation}
\tau^2(y)+a\tau(y)+b=0\label{eq:2ndOrderGeneric}
\end{equation}
where $a$ and $b$ are elements of a field $K$, and $\tau$ is an automorphism on $K$. This tradition of constructing algorithms to compute the Galois groups corresponding to such an equation began with Hendriks in the late 1990s with \cite{HendriksShift,HendriksQ} which compute the difference Galois group of equation~\eqref{eq:2ndOrderGeneric} in the shift and the $q$-difference cases respectively. This tradition continued with Arreche in \cite{Arreche2016} and with Arreche and Zhang in \cite{Arreche2022} to compute the \emph{differential} Galois group of equation~\eqref{eq:2ndOrderGeneric} in the same respective cases, to only name two applications of the differential Galois theory of linear difference equations.

\begin{remark}
    It is a result of the general theory that the differential Galois group of equation~\eqref{eq:2ndOrderGeneric} is in fact a \emph{subgroup} of the difference Galois group of the same equation, see \cite[Proposition~2.8]{Hardouin2008}. As a result, a part of the algorithms to find the former Galois group has been to first compute the latter. The differential Galois theory of linear difference equations in every way builds off of the difference Galois theory developed three decades prior.
\end{remark}

In 2015 Dreyfus and Roques created an algorithm in \cite{DreyfusRoques2015} to compute the difference Galois group of Equation~\eqref{eq:2ndOrderGeneric} over the field of meromorphic functions on an elliptic curve. There is yet to be an algorithm to compute the differential Galois group of Equation~\eqref{eq:2ndOrderGeneric} in the elliptic case.

\subsection{The Method of Residues}

The algorithms in \cite{Arreche2016,Arreche2022} both compute the differential Galois group of equation~\eqref{eq:2ndOrderGeneric} over a function field $\mathcal{K}=C(x)$, where $C$ is a differential closure of $\overline{\mathbb{Q}}$. The automorphisms $\tau$ in these contexts are respectively the shift automorphism and the $q$-difference automorphism
\[\tau(f(x)) = f(x+1);\qquad \tau(f(x))=f(qx).\]
In each case, a critical part of the algorithm requires the ability to determine the summability of arbitrary $f(x)\in C(x)$, that is, the ability to determine the existence of $g(x)\in C(x)$ that would satisfy $\tau(g(x))-g(x)=f(x)$.

This can be done with a unifying concept called \emph{residues}, developed in \cite{Chen2012}. The authors define \defemph{discrete residues} for the shift automorphism and \defemph{$q$-discrete residues} for the $q$-difference automorphism. They obtain the following rather aesthetic results: summability in each context happens if and only if all the corresponding residues vanish, shown as \cite[Proposition~2.5]{Chen2012} and \cite[Proposition~2.10]{Chen2012} respectively. This gives the hope of similar results in other explicit contexts, such as the elliptic case described at the beginning of this introduction.

In fact work has already been done to this effect in the elliptic case: recently \defemph{orbital residues} were developed in \cite{Dreyfus2018} and \cite{Hardouin2021} on the way to the analysis of walks in the quarter plane described in Section~\ref{s:ellipticApplications}. These works also provide preliminary results on summability in terms of orbital residues. However, the results are not a complete obstruction: Proposition~B.8 in \cite{Dreyfus2018} shows that for a meromorphic function on an elliptic curve, its orbital residues all vanish if and only if it is at least ``almost'' summable\footnote{While we will see an analogous result in Chapter~\ref{c:analytic}, the algebraic geometry required to make this statement precise will be set in Chapter~\ref{c:algebraic}.}. The scope of this dissertation is therefore the following two goals.
\begin{enumerate}
    \item[1.] Complete the obstruction by defining what we will call \defemph{panorbital residues}, so that summability occurs if and only if orbital and panorbital residues all vanish.
    \item[2.] Place this summability result in the context of differential Galois theory of linear difference equations, and thus facilitate the creation of the algorithm to find the differential Galois group of equation~\eqref{eq:2ndOrderGeneric} in the elliptic case.
\end{enumerate}

\section{Residues in the Case of Rational Functions}
\label{s:residueExamples}

Here we motivate and summarize the work of \cite{Chen2012} on the development of residues as an obstruction to summability. We will be drawing parallels between this section and future chapters.

In this section, $C$ refers to an algebraically closed field of characteristic $0$. We work with rational functions: elements of the field $\mathcal{K}=C(x)$. Recall, the elements of $\mathcal{K}$ are ratios of polynomials with coefficients in $C$. Finally, as started in the previous section $\shift$ is the rational shift automorphism on $\mathcal{K}$ given by $\tau(f(x)) = f(x+1)$ for all $f(x)\in C(x)$.

The fact that $C$ is algebraically closed gives us the crucial tool of partial fraction decompositions, and the theory of residues is implemented through the lens of such expressions of rational functions. The following result is a formalization of the concept of partial fraction decompositions.

\begin{proposition}
    Let $f(x)\in \mathcal{K}$ be expressed as $a(x)/b(x)$, where the polynomials $a(x),b(x)\in C[x]$ share no linear factors. Let $\mathcal{R}\subset C$ be the set of zeros of $b(x)$. Then there exist constants $c_j(f,\alpha)$ for each $j\in\mathbb{N}$ and $\alpha\in\mathcal{R}$, and a polynomial $g(x)\in C[x]$, such that
    \[f(x) = g(x) + \sum\limits_{j\in\mathbb{N}}\sum\limits_{\alpha\in\mathcal{R}} \frac{c_j(f,\alpha)}{(x-\alpha)^j}.\]
\end{proposition}\medskip\medskip

\noindent The fact that $C$ is algebraically closed is necessary for $b(x)$ to split into linear factors of the form $(x-\alpha)$. The sum is of finite support: if $j$ is greater than the multiplicity of $\alpha$ as a root of $b(x)$ then $c_j(f,\alpha)=0$. For instance, the following is a partial fraction decomposition of a rational function in $\mathbb{C}(x)$:
\[\frac{3 x^6 - 12 x^5 + 11 x^4 + 9 x^3 - 28 x^2 + 39 x - 12}{(x - 2)^2(x^2+1)} = 3x^2-4 + \frac{3}{x-i} + \frac{3}{x+i} - \frac{1}{x-2} + \frac{2}{(x-2)^2}.\]

\begin{remark}
\label{rem:pfdIsGlobal}
    There is motivation to viewing the partial fraction decomposition of a rational function as an aggregation of its local expansions at its poles. In the above example the local expansions of the above function at the poles $i$, $-i$, and $2$ are respectively
    \begin{align*}
        \frac{3}{(x-i)^1} + &\left[3x^2-4 + \frac{3}{x+i} - \frac{1}{x-2} + \frac{2}{(x-2)^2}\right],\\
        \frac{3}{(x+i)^1} + &\left[3x^2-4 + \frac{3}{x-i} - \frac{1}{x-2} + \frac{2}{(x-2)^2}\right],\\
        \frac{2}{(x-2)^2}+\frac{-1}{(x-2)^1} + &\left[3x^2-4 +\frac{3}{x-i}+ \frac{3}{x+i}\right],
    \end{align*}
    with the bracketed expressions having no poles at $i$, $-i$, and $2$ respectively. Each unbracketed expression is called the \defemph{principal part} of the local expansion near that particular pole. In the rational case we have the luxury of this aggregation forming a \emph{global} expression for the rational function. Partial fraction decompositions do not exist in the elliptic context. They will be replaced with more technical tools in Chapters~\ref{c:analytic} and \ref{c:algebraic}.
\end{remark}

The residues defined in \cite{Chen2012} are expressed in terms of such partial fraction decompositions. We first analyze this in the case of the shift automorphism.

\subsection{Discrete Residues}

Define the shift automorphism $\tau:\mathcal{K}\to\mathcal{K}$ by $\tau(f(x)) = f(x+1)$. We are interested in for which rational functions $f(x)\in\mathcal{K}$ there exist $g(x)\in\mathcal{K}$ for which
\[\tau(g(x)) - g(x) = f(x).\]
We call such $f(x)$ \defemph{summable}.

An immediate observation is that all polynomials in $C(x)$ are summable. Indeed, for $n\geq 1$ we have
\[\tau(x^n) - x^n = (x+1)^n  -x^n = \sum\limits_{k=0}^{n-1} \binom{n}{k}x^k,\]
and this collection of polynomials forms a basis of $C[x]$. These polynomials are summable by construction, and because $\tau$ is a field automorphism it follows that all elements of their span $C[x]$ are summable.

The question of whether rational functions are summable is slightly harder. But with sufficient numerical experiments we can see patterns that are suggestive to proper summability criteria. Suppose we construct the following rational function in $\mathbb{C}(x)$ from its partial fraction decomposition:
\[g(x) = \frac{3}{x-1} - \frac{4}{x-2} + \frac{7}{x^2} + \frac{1}{x-\pi}.\]
Then acting $\tau-\Id$ on $g(x)$ yields
\begin{align*}
    \tau(g(x)) - g(x) = \ & \frac{3}{x} - \frac{7}{x-1} + \frac{4}{x-2} \\
    &+ \frac{7}{(x+1)^2}-\frac{7}{x^2}\\
    &+ \frac{1}{x-\pi+1} - \frac{1}{x-\pi}.
\end{align*}
The right-hand side is summable by definition. We have sorted the terms so the poles are grouped by:
\begin{itemize}
    \item Orbits of the underling shift action of $x\mapsto x+1$ on $\mathbb{C}$,
    \item Orders of the poles in each degree.
\end{itemize}
Within each group the numerators sum to zero. This will hold in all experiments, and in fact is exactly the summability criterion in \cite{Chen2012}. These sums are what they refer to as discrete residues, as defined in \cite[Definition~2.3]{Chen2012}.

\begin{definition}
    Let $f(x)\in\mathcal{K}$ have partial fraction decomposition given by
    \[f(x) = g(x) + \sum\limits_{j\in\mathbb{N}}\sum\limits_{\alpha\in\mathcal{R}} \frac{c_j(f,\alpha)}{(x-\alpha)^j},\]
    where $\mathcal{R}$ is the set of poles of $f(x)$, the constants $c_j(f,\alpha)\in C$ are the coefficients of the principal parts of the local expansion of $f(x)$ near $\alpha$, and $g(x)\in C[x]$ is a polynomial. For a $\beta\in C$, let $[\beta]$ be the $\mathbb{Z}$-orbit $\{\beta+n\mid n\in\mathbb{Z}\}$. Then the \defemph{discrete residue} of $f$ at the orbit $[\beta]$ of order $j$ is
    \[\text{dres}(f,[\beta],j) = \sum\limits_{n\in\mathbb{Z}} c_j(f,\beta+n).\]
\end{definition}

These discrete residues then form an obstruction to summability of rational functions, as shown in \cite[Proposition~2.5]{Chen2012}.

\begin{theorem}
    Let $f\in\mathcal{K}$. Then $f$ is summable with respect to the shift automorphism if and only if the discrete residues of $f$ at all orbits and all orders vanish.
\end{theorem}

\subsection{$q$-Discrete Residues}

In this section we still consider the field of rational functions $\mathcal{K}=C(x)$, but now we consider the $q$-difference automorphism $\tau:\mathcal{K}\to\mathcal{K}$ defined by $\tau(f(x)) = f(qx)$, where $q$ is not a root of unity in $C$. For example,
\[\tau\left( \frac{x+1}{x^2+3x+2}\right) = \frac{qx+1}{q^2x^2+3qx+2}.\]
\indent This new automorphism produces complications to the formulation of residues in terms of a partial fraction decomposition. In the previous subsection, the shift difference automorphism shifted every pole by $1$. Now, the $q$-difference automorphism does not shift any pole at $0$. Furthermore, one can show that nonzero constants in $C$ are no longer summable. And recall the thought experiment in the previous subsection of sorting poles into orbits: the partial fraction decomposition must account for this when the automorphism action on $\mathcal{K}$ sends $x$ to $qx$. In all, one must use the following form of the partial fraction decomposition:
\begin{equation}
f(x) = c + xp_1(x) + \frac{p_2(x)}{x}+\sum\limits_{\alpha\in\mathcal{R}}\sum\limits_{j\in\mathbb{N}}\sum\limits_{\ell\in\mathbb{Z}} \frac{c_j(f,q^{\ell}\cdot\alpha)}{(x-q^{\ell}\cdot\alpha)^j}\label{eq:qdresPFD}
\end{equation}
where $c\in C$ is constant, $p_1(x),p_2(x)\in C[x]$ are polynomials, and $\mathcal{R}$ is a maximal set of poles of $f(x)$ that are in distinct $q^{\mathbb{Z}}$-orbits. The definition of $q$-discrete residues is then the following, from \cite[Definition~2.7]{Chen2012}:

\begin{definition}
    Let $f\in \mathcal{K}$ have a partial fraction decomposition as in Equation~\eqref{eq:qdresPFD}. Then the $q$-discrete residue of $f$ at the $q^{\mathbb{Z}}$-orbit $[\beta]_q$ of multiplicity $j$ is defined to be
    \[\text{qres}(f,[\beta]_q,j) = \sum\limits_{\ell\in\mathbb{Z}} q^{-\ell j}c_j(f,q^{\ell}\cdot\beta).\]
    In addition, the constant $c$ is called the $q$-discrete residue of $f$ at infinity, denoted by $\text{qres}(f,\infty)$.
\end{definition}

These $q$-discrete residues form a complete obstruction to summability for the $q$-difference automorphism, as shown in \cite[Proposition~2.10]{Chen2012}:

\begin{theorem}
    Let $f\in\mathcal{K}$. Then $f$ is summable with respect to the $q$-difference automorphism if and only if the $q$-discrete residues of $f$ at all orbits and all orders vanish, as well as the $q$-discrete residue at $\infty$.
\end{theorem}

This shows that even in a case as simple as rational functions, the summability obstruction for a difference automorphism can become quite involved very quickly. In future chapters we will see the technology of residues ramping up in complexity, but this technology proves to be an effective obstruction to summability in the elliptic context as well as that of rational functions.

\chapter{Summability in the Analytic Setting}
\label{c:analytic}

Now that we have seen examples of residues in the rational case, we can work to adapt them to the elliptic case. This is most easily done from the analytic point of view of doubly periodic meromorphic functions on $\mathbb{C}$. The wealth of information about explicit functions in this context lets us choose from the corpus a particular function, the Weierstrass $\zeta$-function, with which to construct an analogue for the partial fraction decompositions utilized in Chapter~\ref{c:basicResidues}.

\section{Review of Elliptic Functions}
\label{s:analysisReview}

The initial formulation of elliptic curves came from the study of doubly-periodic meromorphic functions on the complex plane. Jacobi proved in 1835 in \cite{Jacobi1835} that triply-periodic meromorphic functions in the complex plane cannot exist without being constant, so doubly-periodic is the most rigid behavior a nonconstant periodic function on $\mathbb{C}$ can take. We call these doubly-periodic complex functions \defemph{elliptic functions}.

Myriad authors have written and rewritten wonderful manuscripts about elliptic functions: no result in this section are new. When left uncited, a result in this section can be found in one of \cite[Chapter~VI]{Silverman2009} or \cite[Chapter~I]{silvermanAdvancedTopicsArithmetic1994}. Any result not explicitly stated in either book will come with a precise citation. For those interested in further edification, \cite{Kirwan1992} and \cite{diamondFirstCourseModular2005} have been a source of inspiration.

Let $\lambda_1,\lambda_2\in\mathbb{C}$ be two nonzero numbers such that $\lambda_1/\lambda_2\not\in\mathbb{R}$.\footnote{By precomposing with a rotation and scaling, one may assume without loss of generality that $\lambda_1=1$ and $\lambda_2$ has positive imaginary part. This assumption can make computations easier in certain contexts.} These two numbers additively generate a lattice
\[\Lambda = \{n\lambda_1+m\lambda_2:n,m\in\mathbb{Z}\}\subset \mathbb{C}.\]
Connecting the points in the lattice via the lines determined by these periods tessellates $\mathbb{C}$ into parallelograms, as shown in Figure~\ref{f:fundParallel}.

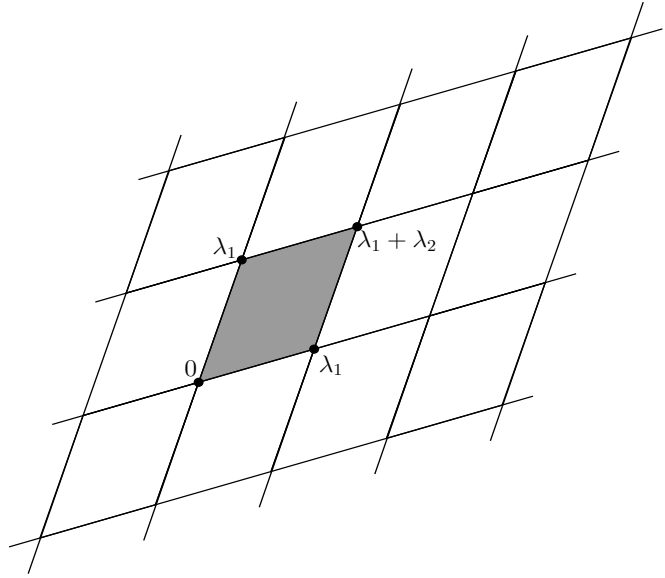
\begin{figure}
  \centering

\tikzset{every picture/.style={line width=0.75pt}}
\resizebox{3.5in}{3in}{
\begin{tikzpicture}[x=0.7pt,y=0.7pt,yscale=-1,xscale=1]

\draw  [draw opacity=0][fill={rgb, 255:red, 155; green, 155; blue, 155 }  ,fill opacity=1 ] (243.9,223.07) -- (335.68,194.51) -- (301.42,299.85) -- (209.65,328.41) -- cycle ;
\draw  [color={rgb, 255:red, 0; green, 0; blue, 0 }  ,draw opacity=1 ] (335.68,194.5) -- (427.45,165.93) -- (393.2,271.27) -- (301.42,299.84) -- cycle ;
\draw  [color={rgb, 255:red, 0; green, 0; blue, 0 }  ,draw opacity=1 ] (427.46,165.93) -- (519.23,137.36) -- (484.97,242.7) -- (393.2,271.27) -- cycle ;
\draw  [color={rgb, 255:red, 0; green, 0; blue, 0 }  ,draw opacity=1 ] (152.13,251.64) -- (243.9,223.08) -- (209.64,328.42) -- (117.87,356.98) -- cycle ;
\draw  [color={rgb, 255:red, 0; green, 0; blue, 0 }  ,draw opacity=1 ] (186.39,146.29) -- (278.16,117.73) -- (243.91,223.07) -- (152.13,251.63) -- cycle ;
\draw  [color={rgb, 255:red, 0; green, 0; blue, 0 }  ,draw opacity=1 ] (278.17,117.72) -- (369.94,89.15) -- (335.68,194.49) -- (243.91,223.06) -- cycle ;
\draw  [color={rgb, 255:red, 0; green, 0; blue, 0 }  ,draw opacity=1 ] (369.94,89.15) -- (461.72,60.58) -- (427.46,165.92) -- (335.69,194.49) -- cycle ;
\draw  [color={rgb, 255:red, 0; green, 0; blue, 0 }  ,draw opacity=1 ] (461.72,60.58) -- (553.49,32.01) -- (519.24,137.35) -- (427.46,165.91) -- cycle ;
\draw  [color={rgb, 255:red, 0; green, 0; blue, 0 }  ,draw opacity=1 ] (117.86,357) -- (209.63,328.43) -- (175.38,433.77) -- (83.61,462.34) -- cycle ;
\draw  [color={rgb, 255:red, 0; green, 0; blue, 0 }  ,draw opacity=1 ] (209.64,328.42) -- (301.41,299.86) -- (267.15,405.2) -- (175.38,433.76) -- cycle ;
\draw  [color={rgb, 255:red, 0; green, 0; blue, 0 }  ,draw opacity=1 ] (301.42,299.85) -- (393.19,271.28) -- (358.93,376.62) -- (267.16,405.19) -- cycle ;
\draw  [color={rgb, 255:red, 0; green, 0; blue, 0 }  ,draw opacity=1 ] (393.19,271.28) -- (484.97,242.71) -- (450.71,348.05) -- (358.94,376.62) -- cycle ;
\draw    (287.5,87) -- (243.91,223.06) ;
\draw    (379.26,58.45) -- (335.68,194.51) ;
\draw    (471.04,29.87) -- (427.45,165.93) ;
\draw    (562.82,1.3) -- (519.23,137.36) ;
\draw    (195.72,115.58) -- (152.13,251.64) ;
\draw    (117.87,356.98) -- (74.28,493.04) ;
\draw    (209.63,328.43) -- (166.05,464.49) ;
\draw    (301.41,299.86) -- (257.82,435.92) ;
\draw    (393.19,271.28) -- (349.6,407.34) ;
\draw    (484.97,242.71) -- (441.38,378.77) ;
\draw    (393.2,271.27) -- (509.5,235) ;
\draw    (358.94,376.62) -- (475.24,340.35) ;
\draw    (427.46,165.91) -- (543.76,129.65) ;
\draw    (461.72,60.58) -- (578.02,24.32) ;
\draw    (161.86,153.99) -- (278.16,117.73) ;
\draw    (127.6,259.34) -- (243.9,223.08) ;
\draw    (93.35,364.68) -- (209.65,328.41) ;
\draw    (59.08,470.03) -- (175.38,433.76) ;
\draw  [fill={rgb, 255:red, 0; green, 0; blue, 0 }  ,fill opacity=1 ] (240.48,223.08) .. controls (240.48,221.19) and (242.01,219.66) .. (243.9,219.66) .. controls (245.79,219.66) and (247.32,221.19) .. (247.32,223.08) .. controls (247.32,224.97) and (245.79,226.5) .. (243.9,226.5) .. controls (242.01,226.5) and (240.48,224.97) .. (240.48,223.08) -- cycle ;
\draw  [fill={rgb, 255:red, 0; green, 0; blue, 0 }  ,fill opacity=1 ] (332.27,194.49) .. controls (332.27,192.6) and (333.8,191.07) .. (335.69,191.07) .. controls (337.58,191.07) and (339.11,192.6) .. (339.11,194.49) .. controls (339.11,196.38) and (337.58,197.91) .. (335.69,197.91) .. controls (333.8,197.91) and (332.27,196.38) .. (332.27,194.49) -- cycle ;
\draw  [fill={rgb, 255:red, 0; green, 0; blue, 0 }  ,fill opacity=1 ] (206.22,328.43) .. controls (206.22,326.54) and (207.75,325.01) .. (209.63,325.01) .. controls (211.52,325.01) and (213.05,326.54) .. (213.05,328.43) .. controls (213.05,330.32) and (211.52,331.85) .. (209.63,331.85) .. controls (207.75,331.85) and (206.22,330.32) .. (206.22,328.43) -- cycle ;
\draw  [fill={rgb, 255:red, 0; green, 0; blue, 0 }  ,fill opacity=1 ] (297.99,299.86) .. controls (297.99,297.97) and (299.52,296.44) .. (301.41,296.44) .. controls (303.3,296.44) and (304.83,297.97) .. (304.83,299.86) .. controls (304.83,301.75) and (303.3,303.28) .. (301.41,303.28) .. controls (299.52,303.28) and (297.99,301.75) .. (297.99,299.86) -- cycle ;
\draw (196.76,308.29) node [anchor=north west][inner sep=0.75pt]    {{\large $0$}};
\draw (303.42,303.24) node [anchor=north west][inner sep=0.75pt]    {{\large $\lambda _{1}$}};
\draw (218.99,202.04) node [anchor=north west][inner sep=0.75pt]    {{\large $\lambda _{1}$}};
\draw (333.59,196.17) node [anchor=north west][inner sep=0.75pt]    {{\large $\lambda _{1} +\lambda _{2}$}};

\end{tikzpicture}
}

  \caption{A lattice $\Lambda$ with its fundamental parallelogram labeled and shaded}
  \label{f:fundParallel}
\end{figure}

Given a meromorphic function $f:\mathbb{C}\to\mathbb{C}$ with periods $\lambda_1$ and $\lambda_2$, its global behavior is completely determined by its behavior on any of these parallelograms. The parallelogram with vertices $0$, $\lambda_1$, $\lambda_2$, and $\lambda_1+\lambda_2$ is called a \defemph{fundamental parallelogram} of the lattice $\Lambda$. Identifying opposite sides of this parallelogram lets us view it as a quotient group (indeed, a torus!) $\mathbb{C}/\Lambda$, so $f$ acting on $\mathbb{C}$ induces a meromorphic function on this torus.

\begin{remark}
    More precisely, a fundamental parallelogram is any translation of the above parallelogram by an element of $\mathbb{C}$.
\end{remark}

A meromorphic function with periods $\lambda_1$ and $\lambda_2$ is called \defemph{elliptic} with respect to $\Lambda$. The collection of elliptic functions with respect to $\Lambda$ forms a field under addition and multiplication, and we denote this field by $\fieldC$.

Any nonconstant elliptic function $f\in\fieldC$ cannot be entire, because its double-periodicity would imply that it is bounded, producing a contradiction to Liouville's Theorem. Therefore there must exist poles of such a function. With some complex contour integration one can prove that the zeros and poles of $f$ must play a balancing game within the fundamental parallelogram. This result is expressed in terms of the usual orders and residues\footnote{There is rhetorical incongruency between the usage of ``residue'' to refer to objects that detect summability, which are ultimately numbers extracted from a representation of a function, and the usage of ``residue'' in the usual complex analytic sense, which are properly speaking attached to differential forms rather than functions. Throughout this dissertation we will explicitly say when we mean the latter. The reader should assume otherwise that we mean the former.} of complex analysis.

\begin{theorem}
\label{thm:balancingGame}
Let $f\in\fieldC$ be elliptic, and let $D$ be a fundamental parallelogram of $\lambda$. Then:
\begin{enumerate}[label=(\alph*)]
\item $\sum\limits_{a\in D} \text{res}_a(f) = 0$,
\item $\sum\limits_{a\in D} \text{ord}_a(f)=0$,
\item $\sum\limits_{a\in D} a\cdot\text{ord}_a(f)\in\Lambda$.
\end{enumerate}
\end{theorem}

Items (b) and (c) respectively say that the number of zeros and poles within $D$ are the same counting multiplicity, and that the sum of the zeros and poles within $D$ is $0\pmod{\Lambda}$ counting multiplicity. We have an immediate corollary of items (a) and (b):

\begin{corollary}
\label{cor:2poles}
    A nonconstant elliptic function has at least two poles in $\mathbb{C}/\Lambda$.
\end{corollary}

This is in fact where a major obstacle arises in adapting previous obstructions of summability to the elliptic case. In Chapter~\ref{c:basicResidues} we saw examples that utilize the partial fraction decomposition of rational functions. A critical feature of the field of rational functions is the existence of rational functions with only one simple pole. In the elliptic case such functions are nonexistent. This means that our analogue for partial fraction decompositions should have one of the following two properties:

\begin{enumerate}[label=\arabic*.]
\item It must be in terms of elliptic functions with two or more poles in $\mathbb{C}/\Lambda$, or
\item It must be in terms of functions that are \emph{not} elliptic, but are sufficiently related to $\Lambda$.
\end{enumerate}

Each strategy presents its own difficulties, but a commonality is that we no longer have easily accessible prototypes of functions with a small number of poles in $\mathbb{C}/\Lambda$.

\subsection{The Weierstrass Elliptic Functions}

The first literary example of a nonconstant elliptic function is quite commonly the Weierstrass $\wp$ function. This function is manually constructed to have poles of order two at precisely the lattice points in $\Lambda$, and the value is given via a series carefully constructed to satisfy convergence wherever it is defined.

\begin{definition}
The \defemph{Weierstrass elliptic function} The function $\wp_{\Lambda}(z):\mathbb{C}\to\mathbb{C}$ associated with $\Lambda$ is defined by the following series:
\[\wp_{\Lambda}(z) = \frac{1}{z^2}+\sum\limits_{\substack{\lambda\in\Lambda \\ \lambda\neq 0}} \left(\frac{1}{(z-\lambda)^2} - \frac{1}{\lambda^2}\right).\]
\end{definition}

It is a classical result that $\wp_{\Lambda}(z)\in \fieldC$. Furthermore the series definition of $\wp_{\Lambda}(z)$ can be differentiated to obtain a series expression for $\wp_{\Lambda}'(z)$, another elliptic function. These two functions in fact satisfy a polynomial relationship:

\begin{proposition}\label{prop:wpCubic}
For all $z\in\mathbb{C}\backslash\Lambda$,
\[\wp_{\Lambda}'(z)^2 = 4\wp_{\Lambda}(z)^3 - g_2(\Lambda)\wp_{\Lambda}(z)-g_3(\Lambda)\]
where the constants $g_2(\Lambda)$ and $g_3(\Lambda)$ are given by the following Eisenstein series:
\[g_2(\Lambda)=60\sum\limits_{\substack{\lambda\in\Lambda \\[0.2em] \lambda\neq 0}} \frac{1}{\lambda^2},\qquad g_3(\Lambda)=140\sum\limits_{\substack{\lambda\in\Lambda \\[0.2em] \lambda\neq 0}} \frac{1}{\lambda^3}.\]
\end{proposition}
This polynomial relationship leads to the \defemph{Weierstrass equation} $y^2=4x^3-g_2(\Lambda)x-g_3(\Lambda)$. We will see this sort of cubic equation again in the algebraic perspective. Indeed, Weierstrass equations are the \emph{starting point} of that direction of inquiry!

An incredible fact is that the two examples discussed thus far are all that is required to formulate the entire field $\fieldC$ of elliptic functions. That is, all elliptic functions in $\fieldC$ can be characterized simply using $\wp_{\Lambda}(z)$ and $\wp_{\Lambda}'(z)$ as building blocks. Precisely stated:

\begin{proposition}
Any $f\in \fieldC$ can be expressed uniquely as
\[A(\wp_{\Lambda}(z)) + \wp_{\Lambda}'(z)\cdot B(\wp_{\Lambda}(z)),\]
where $A$ and $B$ are rational functions over $\mathbb{C}$. Conversely, all such expressions are elliptic functions. \end{proposition}

Consequently $\fieldC$ can be expressed as $\mathbb{C}(\wp_{\Lambda})[\wp_{\Lambda}']$.

For our purposes we require a function that is not actually elliptic. This function is still related to $\wp_{\Lambda}(z)$ in a way that facilitates a method of expression for all elliptic functions. Once again one can show that the series definition in fact converges on $\mathbb{C}\backslash\Lambda$.

\begin{definition}
    The \defemph{Weierstrass $\zeta$-function} $\zeta_{\Lambda}(z):\mathbb{C}\to\mathbb{C}$ associated with $\Lambda$ is defined by the following series:
    \[\zeta_{\Lambda}(z) = \frac{1}{z}+\sum\limits_{\substack{\lambda\in\Lambda \\[0.2em] \lambda\neq 0}} \left(\frac{1}{z-\lambda} + \frac{1}{\lambda}+ \frac{z}{\lambda^2}\right).\]
\end{definition}

\begin{remark}
    Because $\Lambda$ is taken to be fixed and arbitrary throughout this dissertation, we make the practice of dropping the subscript $\Lambda$ from both $\wp(z)$ and $\zeta(z)$.
\end{remark}

From the construction of $\zeta(z)$ there are several immediate observations.

\begin{proposition}
\label{prop:zetaProps}
The function $\zeta(z)$ satisfies the following properties:
    \begin{enumerate}[label=(\alph*)]
    \item $\zeta'(z)=-\wp(z)$.
    \item $\zeta(z)$ has only poles of order $1$ located at the lattice points in $\Lambda$.
    \item There exists an additive group homomorphism $\eta:\Lambda\to\mathbb{C}$ associated to $\Lambda$ called the \defemph{quasi-period map} that satisfies
    \[\zeta(z+\lambda) = \zeta(z) + \eta(\lambda)\]
    for all $\lambda\in\Lambda$.
    \end{enumerate}
\end{proposition} \medskip

The property (a) suggests that the Weierstrass $\zeta$-function can be used to construct all elliptic functions as well. Property (b) and Corollary~\ref{cor:2poles} confirm that $\zeta(z)$ is indeed not an elliptic function. Property (c) will be useful for computations that follow. It also implies the following fact that is crucial for later steps in our reduction process.

\begin{corollary}
\label{cor:zetaDifferenceIsElliptic}
    For any constant $\alpha\in\mathbb{C}$, the function $\zeta(z+\alpha) - \zeta(z)$ is an elliptic function.
\end{corollary}

\begin{proof}
    The function is $\lambda$-periodic for all $\lambda\in\Lambda$:
    \begin{align*}
        & \zeta(z+\lambda+\alpha) - \zeta(z+\lambda)\\
        = \ & \big[\zeta(z+\alpha)-\zeta(z)\big] + \eta(\lambda) - \eta(\lambda) \\
        = \ & \zeta(z+\alpha)-\zeta(z).
    \end{align*}
\end{proof}

\subsection{$\zeta$-Expansions of Elliptic Functions}
\label{s:zetaExp}

We recall that our immediate goal is to find an analogue for the partial fraction decompositions used in Chapter~\ref{c:basicResidues} that enables us to determine summability. Those partial fraction decompositions were expressed in terms of functions with poles of order $1$, so there should also exist expressions for elliptic functions using the Weierstrass $\zeta$-function.

One view of a partial fraction decomposition over $\mathbb{C}$ is that it is an aggregation of all local expansions of a rational function. An expression for an arbitrary $f\in \fieldC$ in terms of $\zeta(z)$ could probably be constructed by using the series expansion of $\zeta(z)$ to mimic local expansions at all poles of $f$, and then aggregating them in some way to produce the function $f$ in a way that guarantees the expression to be elliptic.

Because $\zeta(z)$ follows quasiperiodic behavior as per Proposition~\ref{prop:zetaProps}, the expressions in terms of $\zeta(z)$ are highly nonunique. This will be a mainstay of computations to follow.

\begin{proposition}
\label{prop:zetaExp}
    Let $f(z)\in\fieldC$ be any elliptic function, and let $\{\alpha_1,\dots,\alpha_n\}$ be a set of poles of $f(z)$ such that no two are equivalent modulo $\Lambda$, but any pole of $f(z)$ is equivalent to $\alpha_k$ modulo $\Lambda$ for some $1\leq k\leq n$. Then there exist unique constants $c_0(f)$ and $c_j(f,\alpha_k)$ for $1\leq k\leq n$ and $j\geq 1$ such that $f(z)$ can be expressed as
    \begin{equation}
        \label{eq:zetaExp}
            f(z) = c_0(f)+\sum\limits_{k=1}^n\sum\limits_{j\geq 1}\frac{(-1)^{j-1}c_j(f,\alpha_k)}{(j-1)!}\zeta^{(j-1)}(z-\alpha_k).
    \end{equation}
    We call this a \defemph{$\zeta$-expansion} of $f(z)$.
\end{proposition} \medskip \medskip

This is a classical result whose proof, while seldom recorded, can be found in \cite[pp.~449-450]{whittakerCourseModernAnalysis2006}. We recall the proof here.

\begin{proof}
    We show that such a $\zeta$-expansion can be constructed if we let $c_j(f,\alpha_k)$ be the coefficient of the order $j$ term of the principal part of $f(z)$ near $\alpha_k$, for $j\geq 1$ and $1\leq k\leq n$.

    For the terms on the right-hand side, if $j\geq 2$ then $\zeta^{(j-1)}(z-\alpha_k)$ is in fact elliptic with principal part having a single term that has a pole of order $j$. Thus if $c_j(f,\alpha_k)$ agrees with the coefficients of the principal part of $f$ near $\alpha_k$ for poles of order at least $2$, then the higher-order terms cancel and it suffices to assume that $f$ only has simple poles. In this case, we wish to show that there is a constant $c_0(f)$ such that
    \[f(z) = c_0(f) + \sum\limits_{k=1}^{n} c_1(f,\alpha_k)\zeta(z-\alpha_k).\]
    If $c_1(f,\alpha_k)$ agrees with the principal part of $f(z)$ near $\alpha_k$, then they are the classical residues of $f(z)$ from complex analysis, and from Theorem~\ref{thm:balancingGame} these sum to zero. Then the right-hand side is indeed elliptic: for any $\lambda\in\Lambda$,
    \begin{align*}
    \sum\limits_{k=1}^{n} c_1(f,\alpha_k)\zeta(z+\lambda - \alpha_k) &= \sum\limits_{k=1}^{n} c_1(f,\alpha_k)\cdot \eta(\lambda) + \sum\limits_{k=1}^{n} c_1(f,\alpha_k)\zeta(z-\alpha_k) \\
    &= 0 + \sum\limits_{k=1}^{n} c_1(f,\alpha_k)\zeta(z-\alpha_k).
    \end{align*}
    Therefore the difference
    \[f(z) - \sum\limits_{k=1}^{n} c_1(f,\alpha_k)\zeta(z-\alpha_k)\]
    is elliptic and entire, and therefore constant from Corollary~\ref{cor:2poles}. Thus such a constant $c_0(f)$ exists, finishing the proof.
\end{proof}

Proposition~\ref{prop:zetaExp} shows that the coefficients of all the local expansions of an elliptic function $f(z)$ can in fact be aggregated into a global expression for $f(z)$, mirroring Remark~\ref{rem:pfdIsGlobal}. This will be the sort of construct we will reduce to find a summability criterion for elliptic functions.

\begin{remark}
    Expressions of elliptic functions in terms of the Weierstrass $\zeta$-function exist sparsely in literature relative to the more favored expressions in terms of $\wp(z)$, or even as products of the Weierstrass $\sigma$-function (cf. \cite[Proposition~I.5.5]{silvermanAdvancedTopicsArithmetic1994}). These expressions focus more on algebraic data, whereas the analytic nature of Proposition~\ref{prop:zetaExp} has the advantage of the direct connection to local expansions of elliptic functions.
\end{remark}

\section{Finding a Summability Criterion}
\label{s:analyticSummability}

We have laid down all the necessary background on elliptic functions to discover a summability criterion in the analytic setting. Now we can formulate and solve this problem. After we place a difference algebra structure on $\fieldC$ and define preliminary terms, we describe a process of reducing the $\zeta$-expansion of an arbitrary $f(z)\in\fieldC$ by summable elements. As a result, $f(z)$ will be summable if and only if this reduction is zero.

Yet there is little benefit from a summability criterion that demands the computation of a reduction: it would be of much greater utility to have a criterion computable from $f(z)$ without needing to calculate any reductions. Thus from the reduction we extract particular coefficients that we will declare to be the analytic panorbital residues of $f(z)$. These coefficients will be the key ingredient in our criterion for summability.

\subsection{A Difference Structure on $\mathbb{C}/\Lambda$}
\label{ss:analyticDifferenceStructure}

When formulating the terminology in this subsection we borrow heavily from the work of \cite[Appendix~B]{Dreyfus2018} which formulates orbital residues of elliptic functions in the algebraic setting. What follows is an analytic reinterpretation.

We reaffirm that we are working with the elliptic functions $\fieldC$ that are periodic with respect to the lattice
\[\Lambda=\Span_{\mathbb{Z}}(\lambda_1,\lambda_2)\]
\noindent for a choice of $\lambda_1,\lambda_2\in\mathbb{C}^{\times}$ with $\lambda_1/\lambda_2\not\in\mathbb{R}$. The difference algebra structure of interest is induced by the \defemph{shift automorphism} $\shift$ on $\fieldC$ defined in the following way: for a fixed $\shiftnumber\in\mathbb{C}\backslash \mathbb{Q}\Lambda$, the automorphism $\shift$ acts on a function $f\in\fieldC$ by\footnote{It is easy to conflate the function $f$ being a function on $\mathbb{C}/\Lambda$ and being a function on $\mathbb{C}$. By declaring $\shiftnumber\in\mathbb{C}$, we are implicitly making the choice to view all functions in the following sections as being meromorphic on $\mathbb{C}$, and all elements of $\fieldC$ are considered to be $\Lambda$-periodic meromorphic functions on $\mathbb{C}$.}
\[\shift:f(z)\mapsto f(z+\shiftnumber).\]
The automorphism $\shift$ places a difference field structure on $\fieldC$. We denote the corresponding shift on the points of $\mathbb{C}$ by
\begin{align*}
    \shifttorus&:\mathbb{C}\to\mathbb{C}\\
    \shifttorus&:a\mapsto a-\shiftnumber
\end{align*}
for all $a\in\mathbb{C}$. Thus $\shift^{-1}(f(a))=f(\shifttorus(a))$ for all $a\in\mathbb{C}$ and $f\in\fieldC$. This implies that $\shift$ shifts the poles of an elliptic function in accordance with $\shifttorus$ shifting those poles on the complex plane $\mathbb{C}$. Put precisely:

\begin{lemma}
    A point $p\in\mathbb{C}$ is a pole of order $k$ of $f\in\fieldC$ if and only if $\shifttorus(p)$ is a pole of order $k$ of $\shift(f)\in\fieldC$.
\end{lemma}

Notice the condition placed on $\shiftnumber$ that it not be an element of $\mathbb{Q}\Lambda$: this is equivalent to the automorphism $\shift$ on $\fieldC$ having infinite order. This can also be phrased as $\shiftnumber$ being a non-torsion point of the torus $\mathbb{C}/\Lambda$.

Furthermore, the shift automorphism $\shift$ commutes with differentiation because of the chain rule: this implies that given an elliptic function $f\in\fieldC$, a $\zeta$-expansion of $\shift(f)$ is obtained by applying $\shift$ to the individual terms of the $\zeta$-expansion of $f$.

Finally, the elements of $\fieldC$ that are fixed by $\shift$ are those functions that are periodic with respect to $\lambda_1$, $\lambda_2$, and the non-torsion point $\shiftnumber$, and again by a classical result of Jacobi in \cite{Jacobi1835} this implies that those functions are constant. Thus if $g\in\fieldC$ satisfies $\shift(g)=g$ then $g\in\mathbb{C}$. Moreover, for any $g\in\fieldC\backslash\mathbb{C}$ there does not exist a positive integer $n$ for which $\shift^n(g)=g$.

With this difference structure on $\fieldC$ we may now define summability.

\begin{definition}
    An element $f\in \fieldC$ is called \defemph{summable} if there exists $a\in \fieldC$ such that $f = \shift(a)-a$.
\end{definition}

\begin{example}
    Because $\shift$ is a field automorphism of $\fieldC$ the uniqueness of such a $a$ is up to a term fixed by $\shift$, and as discussed before those terms are constants:
    \begin{alignat*}{2}
     && \shift(a_1)-a_1 &= \shift(a_2)-a_2\\
     \Longleftrightarrow\quad && \shift(a_1-a_2) &= a_1-a_2\\
     \Longleftrightarrow\quad && a_1-a_2 &\in\mathbb{C}.
    \end{alignat*}
\end{example} \medskip

\begin{example}
\label{ex:shiftn}
    Summability is equivalent to being in the image of $\fieldC$ under the operator $\shift-\Id$. The following telescoping series computation shows that the image of $\fieldC$ under $\shift^n-\Id$ consists of summable functions:
    \[\shift^n(a) - a = \shift\left[\sum\limits_{k=0}^{n-1}\shift^k(a)\right] - \sum\limits_{k=0}^{n-1}\shift^k(a).\]
    This technique will feature in the reduction process. It also does not require the analytic structure of $\fieldC$, so it is applicable to general difference fields.
\end{example}

Our reduction process will focus on cancelling poles of $f$ by adding summable elements of $\fieldC$ to $f$. The poles of a summable elliptic function $\shift(a)-a$ for $a\in\fieldC$ are related by $\shifttorus$ orbits, so we require language to analyze the poles on these orbits.

Because $\shiftnumber$ is a non-torsion point of $\Lambda$, the elements of an individual orbit of $\shifttorus$ are all distinct modulo $\Lambda$. Any elliptic function $f\in\fieldC$ has only finitely many poles in $\mathbb{C}/\Lambda$, so there exist only finitely many poles of $f$ in a given orbit of $\shifttorus$. Therefore it is possible for us to quantify the spread of the poles in a given orbit.

\begin{definition}
    Let $0\neq f\in \fieldC$.
    \begin{enumerate}[label=(\alph*)]
    \item The polar dispersion of $f$ is the maximum integer $n$ such that there exist poles $p$ and $p+ns$ of $f$, and is denoted $\pdisp(f)$.
    \item The weak polar dispersion of $f$ is the maximum integer $n$ such that there exist poles $p$ and $p+ns$ of $f$ with orders both at least $2$, and is denoted by $\wpdisp(f)$.
    \end{enumerate}
\end{definition}

\begin{remark}
    The reasons for introducing the notation $\shift^*$ in place of the simpler ``subtract $s$'' are twofold. First, the computations to follow will be cast in terms of orbits of the action of $\shift^*$ on the torus $\mathbb{C}/\Lambda$. Second, this setup is more natural in the more technical algebraic setting in Chapter~\ref{c:algebraic}, and introducing it now sets up the analogy.
\end{remark}

The measure of polar dispersion is crucial to the study of summability in view of the observation that the $\shift-\Id$ operator increments the polar dispersion and weak polar dispersion of a function.

\begin{lemma}
\label{lem:incrementDispersion}
    If $f = \shift(a)-a$ for some nonconstant $a\in\fieldC$ then
    \[\pdisp(f)-\pdisp(a)=\wpdisp(f)-\wpdisp(a)=1.\]
\end{lemma}

\begin{proof}
    The poles of $\shift(a)$ are $\shift^*(p)=p-s$ for all poles $p$ of $a$, so $\pdisp(f)-\pdisp(a)\leq 1$. If $p$ and $p+ns$ are poles of $a$ where $n=\pdisp(a)$, then $p-s$ and $p+ns$ are poles of $\shift(a)-a$, so $\pdisp(f)-\pdisp(a)\geq 1$. The statement for weak polar dispersion is shown analogously.
\end{proof}

\begin{example}
\label{example:wpNotSummable}
    Lemma~\ref{lem:incrementDispersion} immediately implies that particular functions are not summable. If $a\in\mathbb{C}$ then $\shift(a)-a=0$, and if $a\in\fieldC$ is nonconstant then $\shift(a)-a$ has more poles than $a$. As a consequence, no elliptic function with fewer than two \emph{distinct} poles in $\mathbb{C}/\Lambda$ is summable. Thus nonzero constant functions are not summable, and neither are the Weierstrass elliptic function $\wp(z)$ or its derivatives.
\end{example}

\subsection{Analytic Pinnings}
\label{ss:analyticReduction}

The mechanics of the rational case as in Chapter~\ref{c:basicResidues} are extremely accessible because the (function) field of interest is the field of meromorphic functions on the projective line, which has genus zero. To contrast, the elliptic functions $\fieldC$ are meromorphic functions over the torus $\mathbb{C}/\Lambda$, which is a curve of genus $1$. This introduces such algebraic-geometric computations as Corollary~\ref{cor:2poles}: any nonconstant elliptic function has at least two poles, counting multiplicity. This will complicate the later parts of the reduction process where we work to cancel poles of order $1$.

In Subsection~\ref{ss:reductionSearch} we will compute a reduction of an elliptic function $f$: we construct $a\in\fieldC$ such that the polar dispersion of
\[f-[\shift(a)-a]\]
is minimized. After this is done, summability will be equivalent to this reduction vanishing. The reduction will be performed with respect to the following choice of points on $\mathbb{C}/\Lambda$:

\begin{definition}
An \defemph{analytic pinning} is a collection $\mathcal{R}$ of the following choices: for each $(\Lambda\oplus\mathbb{Z}\shiftnumber)$-orbit $\omega\in\allOmegas$, a representative $q_{\omega}\in\omega$, with the condition that $0\in\mathcal{R}$. We denote the orbit containing $0$ by $\widehat{\omega}$.
\end{definition}

Analytic pinnings will be used throughout to express summations over $\shifttorus$-orbits so they are anchored at a specific location, that is, a summation will be taken over all $\hat{q}+n\shiftnumber$ where $\hat{q}\in\mathcal{R}$ and $n\in\mathbb{Z}$.

The choice of an analytic pinning can in fact be encoded within a $\zeta$-expansion of a elliptic function. Let $f\in\fieldC$ be elliptic, and let $\mathcal{R}$ be an analytic pinning. By virtue of a $\zeta$-expansion of $f$ being guaranteed to exist by Proposition~\ref{prop:zetaExp}, the analytic pinning $\mathcal{R}$ lets us organize such a $\zeta$-expansion by orbits of $\shifttorus$.

\begin{corollary}\label{cor:pinnedZetaExp}
    Given a choice of analytic pinning $\mathcal{R}$ of $f$, there exist constants $c_j(f,\hat{q}+ns)$ for $n\in\mathbb{Z}$ and $j\in\mathbb{N}$, and a constant $C_{\mathcal{R}}(f)\in\mathbb{C}$ depending only on $\mathcal{R}$, such that $f(z)$ has global expansion
    \begin{equation}
        f(z) = C_{\mathcal{R}}(f)+\sum\limits_{\hat{q}\in\mathcal{R}} \sum\limits_{n\in\mathbb{Z}}\sum\limits_{j\in\mathbb{N}} \frac{c_j(f,\hat{q}+n\shiftnumber)(-1)^{j-1}}{(j-1)!}\cdot\frac{d^{j-1}}{dz^{j-1}} \zeta(z-\hat{q}-n\shiftnumber).\label{eq:pinnedZetaExp}
    \end{equation}
\end{corollary} \medskip

Indeed, this summation has finitely many terms because we are summing over orbits of the shift $\shifttorus$, namely
\[\{q_{\omega}+ns\mid n\in\mathbb{Z}\}\in \mathbb{C}/(\mathbb{Z}\shiftnumber)\]
for each $\omega\in\allOmegas$. Effectively the summation is over a selection of representatives of the orbits in $\mathbb{C}/\Lambda$, so it will have finite support because the elliptic function $f(z)$ has finitely many poles in $\mathbb{C}/\Lambda$.

\subsection{Orbital Residues}

Analogous to the work of \cite{Dreyfus2018,Hardouin2021}, we define analytic orbital residues and demonstrate how they form a partial obstruction to summability.

\begin{definition}
    Let $f\in \fieldC$ be elliptic. For an orbit $\omega\in\allOmegas$, choose a representative $q_{\omega}$. The \defemph{analytic orbital residue} of $f$ at the orbit $\omega$ of order $j\in\mathbb{N}$ is
    \[\ores(f,\omega,j):=\sum_{n\in\mathbb{Z}}c_j(f,q_{\omega}+n\shiftnumber).\]
\end{definition}

With this definition the orbital residues of $f$ in equation~\eqref{eq:pinnedZetaExp} are visible from the $\zeta$-expansion of $f$.

\begin{remark}
    Recall that the values of $c_j(f,\alpha)$ are determined from the principal parts of $f$ at its local expansion near $z=\alpha$, so the definition of orbital residues is not tied to a particular $\zeta$-expansion. Furthermore, they are independent of the choice of representative of $\omega$: the principal parts of $f$ are not affected when moving by lattice periods, and the orbital residues are defined as sums over all shifts by integer multiples of $\shiftnumber$.
\end{remark}\medskip \medskip

Each orbital residue is invariant under the action of $\shift$, implying the following criterion for summability:
\begin{lemma}
\label{lem:analyticOrbitalResidues}
    Let $f\in\fieldC$. If there exists $a\in\fieldC$ such that $f=\shift(a)-a$, then the orbital residues of $f$ all vanish.
\end{lemma}

\begin{proof}
    Let $a$ have $\zeta$-expansion with coefficients $c_j(a,q_{\omega}+n\shiftnumber)$ for all $\omega\in\allOmegas$ and $n\in\mathbb{Z}$. Then $\shift(a)$ has a $\zeta$-expansion with coefficients
    \[c_j(\shift(a),q_{\omega}+n\shiftnumber) = c_j(a,q_{\omega}+(n-1)\shiftnumber).\]
    Consequently the orbital residues of $\shift(a)$ are
    \[\ores(\shift(a),\omega,j) = \sum\limits_{n\in\mathbb{Z}} c_j(\shift(a),q_{\omega}+n\shiftnumber) = \sum\limits_{n\in\mathbb{Z}} c_j(a,q_{\omega}+(n-1)\shiftnumber) = \ores(a,\omega,j).\]
    Orbital residues are additive, so $\ores(f,\omega,j)=0$ for all $\omega\in\allOmegas$ and $j\geq 1$.
\end{proof}

This result mirrors \cite[Proposition~2.5]{Chen2012}, and in fact an algebraic analogue exists as \cite[Proposition~B.8]{Dreyfus2018}. However, this result is only unidirectional. The converse is not true as the following lemma demonstrates. The lemma concerns a small example, but it is by no means a toy--it becomes an important edge case in the summability of elliptic functions!

\begin{lemma}
\label{lem:analyticLPSummable}
    Consider the elliptic function
    \[f(z)=\alpha_1+\alpha_2[\zeta(z+\shiftnumber)-\zeta(z)]\]
    where $\alpha_1,\alpha_2\in\mathbb{C}$ are not both zero.
    \begin{enumerate}[label=(\roman*)]
    \item All the orbital residues of $f$ vanish, and yet $f$ is not summable.
    \item $f'(z)$ is summable.
    \end{enumerate}
\end{lemma}

\begin{proof}
    \begin{enumerate}[label=(\roman*)]
    \item First note that the function is indeed elliptic by Corollary~\ref{cor:zetaDifferenceIsElliptic}. all the orbital residues vanish by construction. Now suppose there did exist $a\in\fieldC$ such that $f = \shift(a)-a$. Then $a$ must be a nonconstant elliptic function, because $f$ is nonzero. Because $f$ has at most simple poles, Lemma~\ref{lem:incrementDispersion} implies that $a$ has at most simple poles. It also implies that $\pdisp(a)=0$. Therefore the elliptic function $a$ has a single simple pole, contradicting Corollary~\ref{cor:2poles}.
    \item The derivative of $f$ is
    \[f'(z) = -\alpha_2\wp(z+\shiftnumber) + \alpha_2\wp(z) = (\shift-\Id)\big(-\alpha_2\wp(z)\big),\]
    which is summable.
    \end{enumerate}
\end{proof}

This shows that orbital residues are not enough to determine summability. We need additional measurements associated with elliptic functions that can somehow detect functions as in Lemma~\ref{lem:analyticLPSummable}.

\subsection{Searching for a Reduction}
\label{ss:reductionSearch}

We now show the complete reduction process for an elliptic function to minimize polar dispersion. This reduction shall be done in three steps:

\begin{enumerate}[label=\arabic*.]
    \item Canceling poles of order higher than $1$,
    \item Canceling order $1$ poles away $\widehat{\omega}$,
    \item Canceling order $1$ poles on $\widehat{\omega}$.
\end{enumerate}

The first step mimics the rational case so closely that the orbital residues fall out. The fact that we are working with meromorphic functions over a genus $1$ curve means that the second and third steps have complications where interesting additional terms appear in the orbit containing $0$.

Let $f\in\fieldC$ be elliptic and let $\mathcal{R}$ be an analytic pinning. Recall the $\zeta$-expansion for $f$ that is guaranteed from Corollary~\ref{cor:pinnedZetaExp}:

    \begin{equation}
        f(z) = C_{\mathcal{R}}(f)+\sum\limits_{\hat{q}\in\mathcal{R}} \sum\limits_{n\in\mathbb{Z}}\sum\limits_{j\in\mathbb{N}} \frac{c_j(f,\hat{q}+n\shiftnumber)(-1)^{j-1}}{(j-1)!}\cdot\frac{d^{j-1}}{dz^{j-1}} \zeta(z-\hat{q}-n\shiftnumber).\label{eq:analyticReductionStart}
    \end{equation}

    \noindent For $j>1$ the function $\dfrac{d^{j-1}}{dz^{j-1}} \zeta(z-\hat{q}-n\shiftnumber)$ is elliptic: indeed it is $-\dfrac{d^{j-2}}{dz^{j-2}} \wp(z-\hat{q}-n\shiftnumber)$. Now consider
    
    \begin{equation}
        \frac{d^{j-1}}{dz^{j-1}} \zeta(z-\hat{q}) - \frac{d^{j-1}}{dz^{j-1}} \zeta(z-\hat{q}-n\shiftnumber) = \shift^n\left[\frac{d^{j-1}}{dz^{j-1}} \zeta(z-\hat{q}-n\shiftnumber)\right] - \frac{d^{j-1}}{dz^{j-1}} \zeta(z-\hat{q}-n\shiftnumber).\label{eq:analyticReductionTool1}
    \end{equation}
    
    \noindent This is summable by Example~\ref{ex:shiftn}. Thus all the higher-order poles can be grouped at the elements of the analytic pinning $\mathcal{R}$ by adding a summable elliptic function to $f$ as follows:
    
    \begin{align}
        & f(z) + \sum\limits_{\hat{q}\in\mathcal{R}} \sum\limits_{n\in\mathbb{Z}}\sum\limits_{j\geq 2} \frac{c_j(f,\hat{q}+n\shiftnumber)}{(j-1)!}\left[\frac{d^{j-1}}{dz^{j-1}} \zeta(z-\hat{q}) - \frac{d^{j-1}}{dz^{j-1}} \zeta(z-\hat{q}-n\shiftnumber)\right] \nonumber\\
        = \ &  C_{\mathcal{R}}(f)+\sum\limits_{\hat{q}\in\mathcal{R}} \sum\limits_{n\in\mathbb{Z}} c_1(f,\hat{q}+n\shiftnumber)\zeta(z-\hat{q}-n\shiftnumber) \nonumber\\
        &+ \sum\limits_{\hat{q}\in\mathcal{R}} \sum\limits_{n\in\mathbb{Z}}\sum\limits_{j\geq 2} \frac{c_j(f,\hat{q}+n\shiftnumber)(-1)^{j-1}}{(j-1)!}\cdot\frac{d^{j-1}}{dz^{j-1}} \zeta(z-\hat{q})\nonumber\\[0.3em]
        = \ & C_{\mathcal{R}}(f)+\sum\limits_{\hat{q}\in\mathcal{R}} \sum\limits_{n\in\mathbb{Z}} c_1(f,\hat{q}+n\shiftnumber)\zeta(z-\hat{q}-n\shiftnumber) \nonumber\\
        &+ \sum\limits_{\omega\in\allOmegas} \sum\limits_{j\geq 2} \frac{\ores(f,\omega,j)(-1)^{j-1}}{(j-1)!}\cdot\frac{d^{j-1}}{dz^{j-1}} \zeta(z-q_{\omega}).\label{eq:analyticReduction1}
    \end{align}
    
    \noindent This new function has at most one pole of higher order on each orbit $\omega\in\allOmegas$, and therefore by logic similar to Example~\ref{example:wpNotSummable} this function can only be summable if these orbital residues at higher orders all vanish. This completes the first step of the reduction. \\

    For the second step of the reduction assume now that $f(z)$ only has simple poles. Thus it has a $\zeta$-expansion of the form

    \[f(z) = C_{\mathcal{R}}(f)+\sum\limits_{\hat{q}\in\mathcal{R}} \sum\limits_{n\in\mathbb{Z}} c_1(f,\hat{q}+n\shiftnumber)\zeta(z-\hat{q}-n\shiftnumber).\]

    \noindent We show a reduction of $f(z)$ so it has at most one pole on each orbit besides $\widehat{\omega}$. Consider the elliptic function
    \[\zeta(z-\hat{q}-n\shiftnumber) - \zeta(z-n\shiftnumber).\]
    Applying $\shift^n-\Id$ yields the summable elliptic function
    
    \begin{equation}
        \zeta(z-\hat{q}) - \zeta(z) - \zeta(z-\hat{q}-n\shiftnumber) + \zeta(z-n\shiftnumber).\label{eq:analyticReductionTool2}
    \end{equation}
    
    We cancel order $1$ poles away from the orbit $\widehat{\omega}\in\allOmegas$ by adding a summable elliptic function to $f(z)$ as follows:
    
        \begin{align}
        & f(z) + \sum\limits_{\substack{\hat{q}\in\mathcal{R} \\ \hat{q}\neq 0}} \sum\limits_{n\in\mathbb{Z}}
        c_1(f,\hat{q}+n\shiftnumber)\big[\zeta(z-\hat{q}) - \zeta(z) - \zeta(z-\hat{q}-n\shiftnumber) + \zeta(z-n\shiftnumber)\big] \nonumber\\
        = \ &  C_{\mathcal{R}}(f)+ \sum\limits_{n\in\mathbb{Z}} c_1(f,n\shiftnumber)\zeta(z-n\shiftnumber) + \sum\limits_{\substack{\hat{q}\in\mathcal{R} \nonumber\\ \hat{q}\neq 0}} \sum\limits_{n\in\mathbb{Z}} c_1(f,\hat{q}+n\shiftnumber)\zeta(z-\hat{q}) \nonumber\\
        & - \sum\limits_{\substack{\hat{q}\in\mathcal{R} \\ \hat{q}\neq 0}} \sum\limits_{n\in\mathbb{Z}} c_1(f,\hat{q}+n\shiftnumber)\zeta(z) + \sum\limits_{\substack{\hat{q}\in\mathcal{R} \\ \hat{q}\neq 0}} \sum\limits_{n\in\mathbb{Z}} c_1(f,\hat{q}+n\shiftnumber)\zeta(z-n\shiftnumber)  \nonumber
        \end{align}

        \begin{align}
        = \ &  C_{\mathcal{R}}(f)+ \sum\limits_{n\in\mathbb{Z}} c_1(f,n\shiftnumber)\zeta(z-n\shiftnumber) + \sum\limits_{\substack{\omega\in\allOmegas \\ \omega\neq\widehat{\omega}}} \ores(f,\omega,1)\zeta(z-q_{\omega}) \nonumber\\
        & - \sum\limits_{\substack{\allOmegas \\ \omega\neq\widehat{\omega}}} \ores(f,\omega,1)\zeta(z) + \sum\limits_{\substack{\hat{q}\in\mathcal{R} \\ \hat{q}\neq 0}} \sum\limits_{n\in\mathbb{Z}} c_1(f,\hat{q}+n\shiftnumber)\zeta(z-n\shiftnumber)  \nonumber
        \end{align}
        
        \begin{align}
        = \ &  C_{\mathcal{R}}(f) + \sum\limits_{\hat{q}\in\mathcal{R}} \sum\limits_{n\in\mathbb{Z}} c_1(f,\hat{q}+n\shiftnumber)\zeta(z-n\shiftnumber) \nonumber\\
        & + \sum\limits_{\substack{\omega\in\allOmegas \\ \omega\neq\widehat{\omega}}} \ores(f,\omega,1)\zeta(z-q_{\omega}) - \sum\limits_{\substack{\omega\in\allOmegas \\ \omega\neq\widehat{\omega}}} \ores(f,\omega,1)\zeta(z) \label{eq:analyticReduction2}
    \end{align}

    \noindent In this second reduction step we once again have orbital residues congregating at representatives in $\mathcal{R}$, but these orbital residues also congregate at $0$ by virtue of being coefficients of $\zeta(z-0)$. Furthermore, we witness the coefficients from other orbits falling onto the elements of the set $\mathbb{Z}\shiftnumber$.

    We move to the third and final reduction step. Here we assume that $f$ only has simple poles and they are all in $\mathbb{Z}\shiftnumber$. We suppose that $f(z)$ has $\zeta$-expansion
    \[f(z) = C_{\mathcal{R}}(f) + \sum\limits_{\hat{q}\in\mathcal{R}} \sum\limits_{n\in\mathbb{Z}} c_1(f,\hat{q}+n\shiftnumber)\zeta(z-n\shiftnumber),\]
    to continue our reduction from Equation~\eqref{eq:analyticReduction2}. Once again we must handle differences of $\zeta$ functions like in the previous step, but now we must adhere to all poles being in $\mathbb{Z}\shiftnumber$ to not undo previous work.

    The goal for this third step is not to congregate poles at one particular representative in $\mathcal{R}$. Instead, we congregate poles at $0$ and at $\shiftnumber$. There are two analogous methods that we must use at each potential pole $n\shiftnumber$, depending on whether $n\geq 2$ or $n\leq -1$. For $n\geq 2$ the following function is elliptic and summable:
    \[\zeta(z-(n-1)\shiftnumber) - \zeta(z-n\shiftnumber) - \zeta(z) + \zeta(z-\shiftnumber) = (\tau-\Id)[\zeta(z-n\shiftnumber) - \zeta(z-\shiftnumber)]\]

    This gives rise to the following telescoping series for all $n\geq 2$:
    \begin{multline}
        \sum\limits_{k=2}^{n} \left[\zeta(z-(k-1)\shiftnumber) - \zeta(z-k\shiftnumber) - \zeta(z) + \zeta(z-\shiftnumber)\right] \\ = -(n-1)\zeta(z) + n \zeta(z-\shiftnumber) - \zeta(z-n\shiftnumber)\label{eq:analyticReductionTool3}
    \end{multline}

    Therefore we have the following partial reduction:
    \begin{align}
        & f(z) + \sum\limits_{\hat{q}\in\mathcal{R}} \sum\limits_{n\geq 2} c_1(f,\hat{q}+n\shiftnumber)\big[-(n-1)\zeta(z) + n \zeta(z-\shiftnumber) - \zeta(z-n\shiftnumber)\big] \nonumber\\
        = \ & C_{\mathcal{R}}(f) + \sum\limits_{\hat{q}\in\mathcal{R}} \bigg\{  \sum\limits_{n\leq 1} c_1(f,\hat{q}+n\shiftnumber)\zeta(z-n\shiftnumber) - \sum\limits_{n\geq 2} n\cdot c_1(f,\hat{q}+n\shiftnumber) \zeta(z)\nonumber\\
        &  + \sum\limits_{n\geq 2} c_1(f,\hat{q}+n\shiftnumber) \zeta(z) + \sum\limits_{n\geq 2} n\cdot c_1(f,\hat{q}+n\shiftnumber) \zeta(z-\shiftnumber)\bigg\} \label{eq:analyticReduction3.1}
    \end{align}

    Now for $n\leq -1$ the following function is elliptic and summable:    
    \begin{multline}
        \zeta(z-(n+1)\shiftnumber) - \zeta(z-n\shiftnumber) + \zeta(z) - \zeta(z-\shiftnumber) \\ = (\tau-\Id)[\zeta(z-\shiftnumber) - \zeta(z-(n+1)\shiftnumber)]\quad\text{ for }n\leq -1.
    \end{multline}
    The following series telescopes for $n\leq -1$ (be wary of signs):
    \begin{multline}
    \sum\limits_{k=n}^{-1}[\zeta(z-(k+1)\shiftnumber) - \zeta(z-k\shiftnumber) + \zeta(z) - \zeta(z-\shiftnumber)]  \\ = -\zeta(z-n\shiftnumber) + (-n+1)\zeta(z) + n\zeta(z-\shiftnumber)\label{eq:analyticReductionTool4}
    \end{multline}

    We now complete the reduction, importing the previous expressions from Equation~\eqref{eq:analyticReduction3.1}.

    \begin{align}
        & f(z) + \sum\limits_{\hat{q}\in\mathcal{R}} \sum\limits_{n\geq 2} c_1(f,\hat{q}+n\shiftnumber)\big[-(n-1)\zeta(z) + n \zeta(z-\shiftnumber) - \zeta(z-n\shiftnumber)\big] \nonumber\\
        &+ \sum\limits_{\hat{q}\in\mathcal{R}} \sum\limits_{n\leq -1} c_1(f,\hat{q}+n\shiftnumber)\big[-\zeta(z-n\shiftnumber) + (-n+1)\zeta(z) + n\zeta(z-\shiftnumber)\big] \nonumber\\
        = \ & C_{\mathcal{R}}(f) + \sum\limits_{\hat{q}\in\mathcal{R}} \bigg\{  \sum\limits_{n\leq 1} c_1(f,\hat{q}+n\shiftnumber)\zeta(z-n\shiftnumber) - \sum\limits_{n\geq 2} n\cdot c_1(f,\hat{q}+n\shiftnumber) \zeta(z)\nonumber\\
        &  + \sum\limits_{n\geq 2} c_1(f,\hat{q}+n\shiftnumber) \zeta(z) + \sum\limits_{n\geq 2} n\cdot c_1(f,\hat{q}+n\shiftnumber) \zeta(z-\shiftnumber)\bigg\} \nonumber\\
        &+ \sum\limits_{\hat{q}\in\mathcal{R}} \bigg\{ -\sum\limits_{n\leq -1} c_1(f,\hat{q}+n\shiftnumber)\zeta(z-ns) - \sum\limits_{n\leq -1} n\cdot c_1(f,\hat{q}+n\shiftnumber)\zeta(z) \nonumber\\
        &+ \sum\limits_{n\leq -1} c_1(f,\hat{q}+n\shiftnumber)\zeta(z) + \sum\limits_{n\leq -1} n\cdot c_1(f,\hat{q}+n\shiftnumber)\zeta(z-\shiftnumber)\bigg\} \nonumber\\
        = \ & C_{\mathcal{R}}(f) + \sum\limits_{\hat{q}\in\mathcal{R}} \bigg\{  \sum\limits_{n=0}^{1} c_1(f,\hat{q}+n\shiftnumber)\zeta(z-n\shiftnumber) - \sum\limits_{\substack{n\in\mathbb{Z} \\ n\not\in\{0,1\}}} n\cdot c_1(f,\hat{q}+n\shiftnumber) \zeta(z)\nonumber\\
        &  + \sum\limits_{\substack{n\in\mathbb{Z} \\ n\not\in\{0,1\}}} c_1(f,\hat{q}+n\shiftnumber) \zeta(z) + \sum\limits_{\substack{n\in\mathbb{Z} \\ n\not\in\{0,1\}}} n\cdot c_1(f,\hat{q}+n\shiftnumber) \zeta(z-\shiftnumber)\bigg\} \nonumber
        \end{align}
        \begin{align}
        = \ & C_{\mathcal{R}}(f) + \sum\limits_{\hat{q}\in\mathcal{R}} \bigg\{ - \sum\limits_{n\in\mathbb{Z}} n\cdot c_1(f,\hat{q}+n\shiftnumber) \zeta(z)  + \sum\limits_{n\in\mathbb{Z}} c_1(f,\hat{q}+n\shiftnumber) \zeta(z) \nonumber\\
        & + \sum\limits_{n\in\mathbb{Z}} n\cdot c_1(f,\hat{q}+n\shiftnumber) \zeta(z-\shiftnumber)\bigg\} \nonumber\\
        = \ & C_{\mathcal{R}}(f) + \sum\limits_{\omega\in \allOmegas}\ores(f,\omega,1)\zeta(z) \nonumber\\
        &+ \sum\limits_{\hat{q}\in\mathcal{R}} \sum\limits_{n\in\mathbb{Z}} n\cdot c_1(f,\hat{q}+n\shiftnumber)(\zeta(z-\shiftnumber)-\zeta(z)).
        \label{eq:analyticReduction3.2}
    \end{align}

We are at the end of our reduction of $f(z)$. Now recall Lemma~\ref{lem:analyticLPSummable}: any nonzero function of the form $\alpha_1+\alpha_2[\zeta(z-\shiftnumber)-\zeta(z)]$ for $\alpha_1,\alpha_2\in\mathbb{C}$ is not summable. Thus if the orbital residues of $f$ all vanish, then $f$ is summable if and only if the remaining expression
\[C_{\mathcal{R}}(f) + \sum\limits_{\hat{q}\in\mathcal{R}} \sum\limits_{n\in\mathbb{Z}} n\cdot c_1(f,\hat{q}+n\shiftnumber)(\zeta(z-\shiftnumber)-\zeta(z))\]
vanishes. These coefficients are precisely what we must extract to complete our obstruction.

\begin{definition}
\label{def:analyticPanorbital}
    The \defemph{analytic panorbital residues} of $f\in \fieldC$ relative to $\mathcal{R}$ of orders $0$ and $1$ are, respectively, \vspace{-.1in}\begin{gather*}\vphantom{\frac{a}{b}}\pano_\mathcal{R}(f,0):=C_{\mathcal{R}}(f) \qquad\text{and}\\ \pano_\mathcal{R}(f,1):=\smash{\sum_{\hat{q}\,\in\,\mathcal{R}}\ \sum_{n\,\in\,\mathbb{Z}} \ n\cdot c_1(f,\hat{q}+n\shiftnumber)}\end{gather*}
    \phantom{.}
\end{definition}

\noindent Combining the reductions \eqref{eq:analyticReduction1}, \eqref{eq:analyticReduction2}, and \eqref{eq:analyticReduction3.2} of equation~\eqref{eq:analyticReductionStart} gives us the following complete reduction of $f(z)$.

\begin{theorem}
\label{thm:analyticReduction}
    Let $f\in\fieldC$ be elliptic and let $\mathcal{R}$ be an analytic pinning, and let the following be the $\zeta$-expansion for $f$ guaranteed from Corollary~\ref{cor:pinnedZetaExp}:
    \begin{equation}
        f(z) = C_{\mathcal{R}}(f)+\sum\limits_{\hat{q}\in\mathcal{R}} \sum\limits_{n\in\mathbb{Z}}\sum\limits_{j\in\mathbb{N}} \frac{c_j(f,\hat{q}+n\shiftnumber)(-1)^{j-1}}{(j-1)!}\cdot\frac{d^{j-1}}{dz^{j-1}} \zeta(z-\hat{q}-n\shiftnumber).
    \end{equation}

    Then there exists $g\in\fieldC$ such that
    \begin{align*}
        f - [\shift(g)-g] = \ & \pano_\mathcal{R}(f,0) + \pano_\mathcal{R}(f,1)(\zeta(z-\shiftnumber)-\zeta(z))\\
        &+ \sum\limits_{\omega\in\allOmegas} \sum\limits_{j\geq 1} \frac{\ores(f,\omega,j)(-1)^{j-1}}{(j-1)!}\cdot\frac{d^{j-1}}{dz^{j-1}} \zeta(z-q_{\omega})
    \end{align*}
\end{theorem} \medskip \medskip

\noindent We are now able to state and prove the full analytic analogue of \cite[Proposition~B.8]{Dreyfus2018}.

\begin{corollary}
\label{cor:analyticDHRS}
    Let $f\in\fieldC$ be elliptic, and let $\mathcal{R}$ be an analytic pinning. The following are equivalent:
    \begin{enumerate}[label=(\roman*)]
    \item There exists elliptic $g\in\fieldC$ and constants $a,b\in\mathbb{C}$ such that
    \[f - (\shift(g)-g) = a + b(\zeta(z+\shiftnumber) - \zeta(z)).\]
    \item For all orbits $\omega\in\allOmegas$ and all $j\geq 1$,
    \[\ores(f,\omega,j)=0.\]
    \end{enumerate}
\end{corollary}

\begin{proof}
    In fact $a$ and $b$ are respectively the $0$th and $1$st panorbital residues with respect to the analytic pinning $\mathcal{R}$ by which $f$ is reduced.
\end{proof}

\subsection{Statement of the Criterion Via Panorbital Residues}
\label{ss:analyticStatement}

Theorem~\ref{thm:analyticReduction}, in combination with Lemma~\ref{lem:analyticOrbitalResidues} and Lemma~\ref{lem:analyticLPSummable}, show the following obstruction for summability.

\begin{corollary}
\label{cor:analyticObstruction}
    Let $f\in\fieldC$ be elliptic and let $\mathcal{R}$ be an analytic pinning. Then $f$ is summable if and only if all the analytic orbital residues of $f$ and both the analytic panorbital residues of $f$ relative to $\mathcal{R}$ vanish.
\end{corollary}

\section{Discussion of the Results}

We have found a complete obstruction to summability of elliptic functions. Our methods required the use of an analytic pinning relative to which the obstruction is calculated. An immediate question is how choices of the analytic pinning affect the obstruction of Corollary~\ref{cor:analyticObstruction} and the reduction in Theorem~\ref{thm:analyticReduction}. We discuss this in more detail in Subsection~\ref{ss:choiceOfAnalyticPinning}.

Knowing the obstruction a priori creates the possibility for proofs that circumvent the need for deriving the reduction. We provide such a proof in Subsection~\ref{ss:analyticNiceProof}.

\subsection{Effects of the Choice of Analytic Pinning}
\label{ss:choiceOfAnalyticPinning}

While the orbital residues of an elliptic function are not associated with an analytic pinning $\mathcal{R}$, the panorbital residues require such a choice of $\mathcal{R}$ to even be defined. The effects of this choice can be made concrete by analyzing when two analytic pinnings differ on a single orbit $\omega\in\allOmegas$. This suffices. Even if two analytic pinnings differ on infinitely many orbits, the panorbital residues of any given elliptic function are computed as sums of finite support.

\begin{proposition}
\label{prop:changingAnalyticPinnings}
    Let $\mathcal{R}_1$ and $\mathcal{R}_2$ be two analytic pinnings with the same representatives, except for exactly one orbit $\omega\in\allOmegas$ for which $q_{\omega}\in\mathcal{R}_1$ has been replaced with $q_{\omega}+k\shiftnumber+\lambda\in\mathcal{R}_2$ where $k\in\mathbb{Z}$ and $\lambda\in\Lambda$. Let $f\in\fieldC$ be elliptic. Then:
    \begin{enumerate}[label=\arabic*.]
    \item $\pano_{\mathcal{R}_2}(f,0) = \pano_{\mathcal{R}_1}(f,0) + \eta(\lambda)\cdot\ores(f,\omega,1)$,
    \item $\pano_{\mathcal{R}_2}(f,1) = \pano_{\mathcal{R}_1}(f,1) -k\cdot\ores(f,\omega,1)$.
    \end{enumerate}
\end{proposition} \medskip

\begin{proof}
    Take the difference of the two $\zeta$-expansions of $f$ relative to $\mathcal{R}_1$ and $\mathcal{R}_2$ that are guaranteed from Corollary~\ref{cor:pinnedZetaExp}.
        \begin{align*}
        f(z) - f(z)  = \ & C_{\mathcal{R}_2}(f)+\sum\limits_{\hat{q}\in\mathcal{R}_2} \sum\limits_{n\in\mathbb{Z}}\sum\limits_{j\in\mathbb{N}} \frac{c_j(f,\hat{q}+n\shiftnumber)(-1)^{j-1}}{(j-1)!}\cdot\frac{d^{j-1}}{dz^{j-1}} \zeta(z-\hat{q}-n\shiftnumber) \\
        & - C_{\mathcal{R}_1}(f)+\sum\limits_{\hat{q}\in\mathcal{R}_1} \sum\limits_{n\in\mathbb{Z}}\sum\limits_{j\in\mathbb{N}} \frac{c_j(f,\hat{q}+n\shiftnumber)(-1)^{j-1}}{(j-1)!}\cdot\frac{d^{j-1}}{dz^{j-1}} \zeta(z-\hat{q}-n\shiftnumber) \\
        = \ & C_{\mathcal{R}_2}(f) - C_{\mathcal{R}_1}(f) \\
        &+ \sum\limits_{n\in\mathbb{Z}}\sum\limits_{j\in\mathbb{N}} \frac{c_j(f,q_{\omega}+k\shiftnumber+\lambda+n\shiftnumber)(-1)^{j-1}}{(j-1)!}\cdot\frac{d^{j-1}}{dz^{j-1}} \zeta(z-q_{\omega}-k\shiftnumber-\lambda-n\shiftnumber) \\
        & - \sum\limits_{n\in\mathbb{Z}}\sum\limits_{j\in\mathbb{N}} \frac{c_j(f,q_{\omega}+n\shiftnumber)(-1)^{j-1}}{(j-1)!}\cdot\frac{d^{j-1}}{dz^{j-1}} \zeta(z-q_{\omega}-n\shiftnumber).
    \end{align*}
    Because $f$ is elliptic, the coefficients $c_j(f,\alpha)$ of its principal parts are invariant under translates of $\alpha$ by elements of $\Lambda$. Furthermore, the counter variable $n$ in the first summation can be translated by $-k$ to obtain
        \begin{align*}
        0  = \ & C_{\mathcal{R}_2}(f) - C_{\mathcal{R}_1}(f) \\
        &+ \sum\limits_{n\in\mathbb{Z}}\sum\limits_{j\in\mathbb{N}} \frac{c_j(f,q_{\omega}+n\shiftnumber)(-1)^{j-1}}{(j-1)!}\cdot\frac{d^{j-1}}{dz^{j-1}} \zeta(z-q_{\omega}-\lambda-n\shiftnumber) \\
        & - \sum\limits_{n\in\mathbb{Z}}\sum\limits_{j\in\mathbb{N}} \frac{c_j(f,q_{\omega}+n\shiftnumber)(-1)^{j-1}}{(j-1)!}\cdot\frac{d^{j-1}}{dz^{j-1}} \zeta(z-q_{\omega}-n\shiftnumber).
    \end{align*}
    Recall that $\dfrac{d^{j-1}}{dz^{j-1}}\zeta(z)$ is elliptic for all $j\geq 2$. Therefore the two summations cancel for $j\geq 2$ and we are left with
    \[ 0  = \  C_{\mathcal{R}_2}(f) - C_{\mathcal{R}_1}(f) 
        + \sum\limits_{n\in\mathbb{Z}}c_1(f,q_{\omega}+n\shiftnumber)\cdot\zeta(z-q_{\omega}-\lambda-n\shiftnumber) 
         - \sum\limits_{n\in\mathbb{Z}}c_1(f,q_{\omega}+n\shiftnumber)\cdot \zeta(z-q_{\omega}-n\shiftnumber).\]
    The $\zeta$ expressions are related precisely by the quasi-period map $\eta$, implying that
    \[C_{\mathcal{R}_2}(f) - C_{\mathcal{R}_1}(f) = \eta(\lambda)\cdot \sum\limits_{n\in\mathbb{Z}} c_1(f,q_{\omega}+n\shiftnumber) = \eta(\lambda)\cdot\ores(f,\omega,1).\]
    This shows item 1. Item 2 can be shown directly:
    \begin{align*}
        &\pano_{\mathcal{R}_2}(f,1) - \pano_{\mathcal{R}_1}(f,1) \\
        = \ &\sum_{\hat{q}\,\in\,\mathcal{R}_2}\ \sum_{n\,\in\,\mathbb{Z}} \ n\cdot c_1(f,\hat{q}+n\shiftnumber) - \sum_{\hat{q}\,\in\,\mathcal{R}_1}\ \sum_{n\,\in\,\mathbb{Z}} \ n\cdot c_1(f,\hat{q}+n\shiftnumber)\\
        = \ & \sum_{n\,\in\,\mathbb{Z}} \ n\cdot c_1(f,q_{\omega}+k\shiftnumber+\lambda+n\shiftnumber) - \sum_{n\,\in\,\mathbb{Z}} \ n\cdot c_1(f,q_{\omega}+n\shiftnumber)\\
        = \ & \sum_{n\,\in\,\mathbb{Z}} \ (n-k)\cdot c_1(f,q_{\omega}+n\shiftnumber) - \sum_{n\,\in\,\mathbb{Z}} \ n\cdot c_1(f,q_{\omega}+n\shiftnumber)\\
        = \ & -k\sum\limits_{n\in\mathbb{Z}} c_1(f,q_{\omega}+n\shiftnumber)\\
        = \ & -k\cdot\ores(f,\omega,1).
    \end{align*}
\end{proof}
This gives us a precise statement of how the panorbital residues are affected by choice of analytic pinning. An immediate corollary gives us information on when they are \emph{independent} of the analytic pinning.

\begin{corollary}
    Let $f\in\fieldC$ be an elliptic function such that all the orbital residues of $f$ vanish. Then the panorbital residues $\pano_{\mathcal{R}}(f,0)$ and $\pano_{\mathcal{R}}(f,1)$ are unaffected by the choice of $\mathcal{R}$.
\end{corollary}

\subsection{A Tidier Proof of Theorem~\ref{thm:analyticReduction}}
\label{ss:analyticNiceProof}

Now that we have found the required definition of analytic panorbital residues, we can work to show Theorem~\ref{thm:analyticReduction} with their knowledge a priori. We restate the reduction here for convenience before any simplifications: given elliptic $f\in\fieldC$ and analytic pinning $\mathcal{R}$, there exists $a\in\fieldC$ such that
    \begin{align}
        f - [\shift(a)-a] = \ & \pano_\mathcal{R}(f,0) + \pano_\mathcal{R}(f,1)(\zeta(z-\shiftnumber)-\zeta(z))\nonumber\\
        &+ \ores(f,\widehat{\omega},1)\zeta(z) + \sum\limits_{\substack{\omega\in\allOmegas \\\omega\neq\widehat{\omega}}} \ores(f,\omega,1)\zeta(z-q_{\omega}) \nonumber\\
        &+ \sum\limits_{\omega\in\allOmegas} \sum\limits_{j\geq 2} \frac{\ores(f,\omega,j)(-1)^{j-1}}{(j-1)!}\cdot\frac{d^{j-1}}{dz^{j-1}} \zeta(z-q_{\omega})\label{eq:finalReduction}
    \end{align}

    The proof method will involve analyzing the reduction in the following way:
    \begin{enumerate}[label=\arabic*.]
    \item Show that all poles of the above reduction are contained in $\mathcal{R}\cup \{\shiftnumber\}$,
    \item Use the fact that the orbital \emph{and} panorbital residues are invariant under the reduction process to conclude that the above reduction \emph{must} be correct.
    \end{enumerate}

    \begin{proof}
    First, note that under the reduction of Subsection~\ref{ss:analyticReduction} the constant term of the $\zeta$-expansion is never changed, so $\pano_{\mathcal{R}}(f,0)$ is indeed the correct constant coefficient. Furthermore, the reduction is done by adding summable functions, which have vanishing orbital residues from Lemma~\ref{lem:analyticOrbitalResidues}, so the orbital residues remain invariant through every step of the process.
    
    \begin{enumerate}[label=\arabic*.]
        \item In the first step of the reduction, we cancelled poles of order $j\geq 2$ by using the summable elliptic functions
        \[\frac{d^{j-1}}{dz^{j-1}} \zeta(z-q_{\omega}) - \frac{d^{j-1}}{dz^{j-1}} \zeta(z-q_{\omega}-n\shiftnumber)\]
        from Equation~\eqref{eq:analyticReductionTool1}. The orbital residues of these functions are $0$ due to summability and the panorbital residues are $0$ because these functions have $\zeta$-expansions with no terms with lower-order poles. Thus the orbital and panorbital residues of $f$ are not changed in this reduction. Once all higher-order poles of $f$ are accumulated at $f$ the resulting part of the $\zeta$-expansion must indeed be
        \[\sum\limits_{\omega\in\allOmegas} \sum\limits_{j\geq 2} \frac{\ores(f,\omega,j)(-1)^{j-1}}{(j-1)!}\cdot\frac{d^{j-1}}{dz^{j-1}} \zeta(z-q_{\omega}),\]
        which does indeed match the higher order terms of Equation~\eqref{eq:finalReduction}.

        In the second step of the reduction, we cancelled poles of order $1$ on orbits $\omega\in\allOmegas$ excluding $\widehat{\omega}$ by using
        \[\zeta(z-\hat{q}) - \zeta(z) - \zeta(z-\hat{q}-n\shiftnumber) + \zeta(z-n\shiftnumber)\]
        from Equation~\eqref{eq:analyticReductionTool2}. The orbital residues of this summable function vanish, and the panorbital residue is
        \[0\cdot 1 + 0\cdot (-1) + n\cdot(-1) + n\cdot 1 = 0,\]
        so once again the orbital and panorbital residues of $f$ remain invariant under reduction by these terms. Once the poles on $\omega$ are congregated at $q_{\omega}$, the resulting part of the $\zeta$-expansion must be
        \[\sum\limits_{\substack{\omega\in\allOmegas \\ \omega\neq\widehat{\omega}}} \ores(f,\omega,1)\zeta(z-q_{\omega}),\]
        which does indeed match the reduction in Equation~\eqref{eq:finalReduction}.

        The third step cancels poles on $\widehat{\omega}$ so all the remaining poles are at $0$ and $\shiftnumber$. It remains to be shown that the coefficients of $\zeta(z)$ and $\zeta(z-\shiftnumber)$ are correct. The functions used in the reduction are
        \begin{align*}
            -(n-1)\zeta(z) + n \zeta(z-\shiftnumber) - \zeta(z-n\shiftnumber) & \qquad \text{ for }n\geq 2,\\
            -\zeta(z-n\shiftnumber) + (-n+1)\zeta(z) + n\zeta(z-\shiftnumber) & \qquad \text{ for }n\leq -1,
        \end{align*}
        from Equations~\eqref{eq:analyticReductionTool3} and~\eqref{eq:analyticReductionTool4}. These have vanishing orbital residues because they are summable, and their panorbital residues are respectively
        \begin{align*}
            0\cdot (-n+1) + 1\cdot n + n\cdot (-1) &= 0,\\
            n\cdot (-1) + 0\cdot (-n+1) + 1\cdot n&= 0.
        \end{align*}
        Therefore the orbital and panorbital residues of $f$ are invariant through the entire reduction process. At the end of the reduction, the only possible pole away from $\mathcal{R}$ is contributed by $\zeta(z-\shiftnumber)$, so the only contribution to the panorbital residue of the reduction is by the coefficient of this term. As a consequence, the coefficient of $\zeta(z-\shiftnumber)$ must be $\pano_{\mathcal{R}}(f,1)$. Because the orbital residues of the reduction must also agree with those of $f$, the coefficient of $\zeta(z)$ must be $\ores(f,\widehat{\omega},1)-\pano_{\mathcal{R}}(f,1)$. Both of these coefficients appear in Equation~\eqref{eq:finalReduction}. This proves that the reduction in Equation~\eqref{eq:finalReduction} must be correct. 
        \end{enumerate}
    \end{proof}

\section{Example Application: Logarithmic Derivatives}
\label{s:sigmaExample}

We now discuss an application of the results to the \defemph{logarithmic derivative} $\frac{a'}{a}$ of an elliptic function $a\in\fieldC$. Analyzing the $\zeta$-expansion of $\frac{a'}{a}$ using the data of the zeros and poles of $a$ is difficult with no external knowledge. The way we resolve this is by using an expression of $a$ in terms of another function related to the Weierstrass $\wp$ and $\zeta$-functions, namely the \defemph{Weierstrass $\sigma$-function}. The properties of this function, some of which we mention below, are outlined in \cite[Section~1.5]{silvermanAdvancedTopicsArithmetic1994}.

\begin{definition}
    The \defemph{Weierstrass $\sigma$-function} corresponding to a lattice $\Lambda$ is defined by the infinite product
    \[\sigma(z) = z\prod\limits_{\substack{\lambda\in\Lambda \\ \lambda\neq 0}} \left(1-\frac{z}{\lambda}\right)e^{z/\lambda + (1/2)(z/\lambda)^2}.\]
\end{definition}

One can confirm that $\sigma(z)$ is an entire function on $\mathbb{C}$ with simple zeros on $\Lambda$. What is more, it is \emph{directly} linked to the Weierstrass $\zeta$-function in the following manner.

\begin{lemma}
    The logarithmic derivative of $\sigma(z)$ is
    \[\frac{\sigma'(z)}{\sigma(z)} = \zeta(z).\]
\end{lemma}

Finally, similar to $\zeta$-expansions of elliptic functions, there exist expressions of elliptic functions as products of $\sigma$ functions, as stated in \cite[Proposition~5.5]{silvermanAdvancedTopicsArithmetic1994}.

\begin{proposition}
    Consider an elliptic function $a\in\fieldC$. For each $\alpha$ a representative of a coset of $\mathbb{C}/\Lambda$, if $\alpha$ is a zero of $a$ let $m(\alpha)$ be its order, if $\alpha$ is a pole of $a$ let $-m(\alpha)$ be its order, and let $m(\alpha)=0$ otherwise. Let $\alpha_1$, $\dots$, $\alpha_n$ be a complete set of representatives of zeros and poles of $a$, and let
    \[b = \sum\limits_{i=1}^{n} m(\alpha_i)\cdot \alpha_i\in\Lambda.\]
    Then there exists a constant $c\in\mathbb{C}$ such that
    \[a(z) = c\cdot \frac{\sigma(z)}{\sigma(z-b)}\cdot \prod\limits_{i=1}^{n} \sigma(z-\alpha_i)^{m(\alpha_i)}.\]
\end{proposition}

Now choose an analytic pinning $\mathcal{R} = \{q_{\omega}\}$ indexed by $\omega\in\allOmegas$, and express $a\in\fieldC$ as
\[a(z) = c\cdot \frac{\sigma(z)}{\sigma(z-b)}\cdot \prod\limits_{\omega\in\allOmegas}\prod\limits_{n\in\mathbb{Z}} \sigma(z-q_{\omega}-n\shiftnumber)^{m(q_{\omega}+n\shiftnumber)},\]
where $c\in\mathbb{C}$ and
\[b = \sum\limits_{\omega\in\allOmegas}\sum\limits_{n\in\mathbb{Z}} (q_{\omega}+n\shiftnumber)\cdot m(q_{\omega}+n\shiftnumber)\in\Lambda.\]
It follows from the additivity of the logarithmic derivative that the $\sigma$ expression for $a$ above gives rise to a $\zeta$-expansion for $\frac{a'}{a}$:
\begin{align*}\frac{a'}{a} &= \frac{\frac{d}{dz}[c]}{c} + \zeta(z) - \zeta(z-b) + \sum\limits_{\omega\in\allOmegas}\sum\limits_{n\in\mathbb{Z}} m(q_{\omega}+n\shiftnumber)\zeta(z-q_{\omega}-n\shiftnumber)\\
& = \eta(b) + \sum\limits_{\omega\in\allOmegas}\sum\limits_{n\in\mathbb{Z}} m(q_{\omega}+n\shiftnumber)\zeta(z-q_{\omega}-n\shiftnumber),
\end{align*}
where $\eta$ is the quasi-period map of $\Lambda$ as defined in Proposition~\ref{prop:zetaProps}. We immediately see $\pano_{\mathcal{R}}\left(\frac{a'}{a},0\right)=\eta(b)$. The first-order orbital residues of $\frac{a'}{a}$ are
\[\ores\left(\frac{a'}{a},\omega,1\right) = \sum\limits_{n\in\mathbb{Z}} m(q_{\omega}+n\shiftnumber),\]
and there are no higher-order orbital residues. Moreover, the orbital and panorbital residues of $\frac{a'}{a}$ relative to $\mathcal{R}$ are encoded in the value of $b$:
\begin{align*}
b &= \sum\limits_{\omega\in\allOmegas}\sum\limits_{n\in\mathbb{Z}} m(q_{\omega}+n\shiftnumber)\cdot (q_{\omega}+n\shiftnumber) \\
&=  \sum\limits_{\omega\in\allOmegas}\sum\limits_{n\in\mathbb{Z}}ns\cdot  m(q_{\omega}+n\shiftnumber) + \sum\limits_{\omega\in\allOmegas}\sum\limits_{n\in\mathbb{Z}}q_{\omega}\cdot m(q_{\omega}+n\shiftnumber)\\
&=  s\cdot\pano_{\mathcal{R}}\left(\frac{a'}{a},1\right) + \sum\limits_{\omega\in\allOmegas}q_{\omega}\cdot \ores\left(\frac{a'}{a},\omega,1\right).
\end{align*}
In short, the summability of $\frac{a'}{a}$ can be determined completely from the information about the zeros and poles of $a$.

\chapter{Summability in the Algebraic Setting}
\label{c:algebraic}

There is an alternative view of elliptic function fields from a purely algebraic perspective. The construction of these fields can come from multiple directions. A common approach via algebraic geometry is to define an elliptic curve to be a nonsingular projective curve of genus one, along with a choice of base point. This perspective is developed in \cite[Chapter~III]{Silverman2009}. An alternative is to work with discrete valuation rings in an algebraic function field over a base field, and then to formulate purely algebraic analogues of geometric concepts including the genus of a function field and the Riemann-Roch Theorem, as is done in \cite{Stichtenoth2009} and \cite{Goldschmidt2003}. We will not take this second approach: we follow that of \cite{Silverman2009}.

\section{Elliptic Curves as Projective Curves of Genus 1} \label{s:algebraicIntro}

As in Section~\ref{s:analysisReview}, unless otherwise stated, all results in this section come from \cite[Chapter~III]{Silverman2009}. The algebraic theory of elliptic curves has been generalized much farther than those over the field $\mathbb{C}$. Because it is possible to extend the theory at little expense, we concern ourselves here with elliptic curves over a field $K$ of characteristic\footnote{Silverman actually discusses \emph{perfect} fields $K$ in \cite[Chapter~III]{Silverman2009}. By definition fields of characteristic $0$ are perfect: see \cite[Definition~4.16]{aluffiAlgebraChapter02021}.} $0$.

\begin{definition}
\label{def:ellipticCurve}
    An \defemph{elliptic curve} is a pair $(E,O)$ of a nonsingular curve $E$ of genus $1$ and a specified \defemph{base point}  $O\in E$. If $E$ is the variety of a polynomial with coefficients in $K$ then $E$ is said to be \defemph{defined over $K$}.
\end{definition}

Famously, elliptic curves have a commutative group structure on its points known as the \defemph{group law}. We denote this group law by $\oplus$ and $\ominus$, so in our notation the group axioms become:
\begin{itemize}
    \item For every $P,Q\in E$ there exists a unique $P\oplus Q\in E$,
    \item There exists a unique $O\in E$, namely the base point from Definition~\ref{def:ellipticCurve}, that satisfies $P\oplus O = O\oplus P=P$ for all $P\in E$,
    \item For every $P\in E$ there exists a unique $\ominus P\in E$ such that $P\oplus (\ominus P) = \ominus P \oplus P = O$.\footnote{We denote $P\oplus (\ominus P)$ by $P\ominus P$.}
    \item For every $P,Q,R\in E$,
    \[ P\oplus (Q\oplus R) = (P\oplus Q) \oplus R.\]
\end{itemize}
Our shift automorphism will be defined in terms of this group law.

\subsection{Elliptic Functions}

The description of $E$ being a curve of genus $1$ seems loose, but it can be shown that every elliptic curve $E$ defined over $K$ can be viewed as arising from a nonhomogeneous \defemph{Weierstrass equation}
\begin{equation}
Y^2 = 4X^3-g_2X-g_3,\label{eq:weierEqAlg}
\end{equation}
a polynomial equation in variables $X$ and $Y$ with coefficients in $K$ that satisfies the nonsingularity condition $g_2^3\neq 27g_3^2$. In precise terms, there exists a map from $E$ to two-dimensional projective space $\mathbb{P}^2$ over $K$ defined as $\phi=[x:y:1]$ on $E\, \backslash\, \{O\}$ and $\phi(O)=[0:1:0]$, and the map $\phi$ defines an isomorphism between $E$ and the curve defined by the homogenization of Equation~\eqref{eq:weierEqAlg}. The functions $x,y:E\to \mathbb{P}^1$ are meromorphic functions\footnote{A meromorphic function $f:E\to \mathbb{P}^1$ is defined to send $P$ to $[f(P):1]$ if $f$ is regular at $P$, and to $[1:0]$ if $f$ has a pole at $P$, cf. \cite[Example~II.2.2]{Silverman2009}.} on $E$ and are called the \defemph{Weierstrass coordinates} of the elliptic curve. These functions in fact are sufficient to construct the field $\fieldAlg$ of meromorphic functions on $E$.\footnote{We denote the field by $\fieldAlg$ to suggest that we are working with projective curves of genus $1$, a feature that is more prominent in the algebraic setting than the analytic one.}

\begin{proposition}
    The equality of fields $\fieldAlg = K(x,y)$ holds.
\end{proposition}

\noindent As before, we call the elements of $\fieldAlg$ \defemph{elliptic functions} on $E$.

In the case where $K=\mathbb{C}$, the parallels between Equation~\eqref{eq:weierEqAlg} and Proposition~\ref{prop:wpCubic} suggest a correspondence between the analytic and algebraic formulations of elliptic curves over $\mathbb{C}$. This is indeed the case, formulated as the Uniformization Theorem.

\begin{theorem}[Uniformization Theorem]
\label{thm:uniformization}
Let $\mathbb{C}(x,y)$ be an elliptic function field with associated Weierstrass equation $Y^2 = 4X^3  - g_2X-g_3$. Then there exists a lattice $\Lambda\subset \mathbb{C}$, unique up to homothety, and a field isomorphism
\[\varphi:\fieldC\to \mathbb{C}(x,y),\qquad \wp\mapsto x, \ \ \wp'\mapsto y\]
where $\wp$ is the Weierstrass elliptic function with respect to $\Lambda$. \end{theorem}

\begin{proof} See \cite[Corollary~VI.5.1.1]{Silverman2009}. \end{proof}

\noindent Conversely, every lattice $\Lambda$ gives rise to an elliptic function field.

\begin{theorem} Let $\Lambda\subset\mathbb{C}$ be a lattice with corresponding Eisenstein series $g_2(\Lambda)$ and $g_3(\Lambda)$. Let $\fieldAlg=\mathbb{C}(x,y)$ be the elliptic function field arising from the Weierstrass equation $Y^2 = 4X^3  - g_2(\Lambda)X-g_3(\Lambda)$, where $x$ and $y$ are the corresponding Weierstrass coordinate functions. Then the map
\[\varphi:\fieldC\to \mathbb{C}(x,y),\qquad \wp\mapsto x, \ \ \wp'\mapsto y\]
is an isomorphism of fields.
\end{theorem}

\begin{proof} See \cite[Proposition~VI.3.6(b)]{Silverman2009}. \end{proof}

This gives us two views of elliptic curves defined over $\mathbb{C}$. Consider a lattice $\Lambda\subset\mathbb{C}$ and its corresponding Weierstrass elliptic function $\wp(z)$. From Proposition \ref{prop:wpCubic}, $\wp$ and $\wp'$ satisfy the polynomial relationship
\[\wp'(z)^2 = 4\wp(z)^3 - g_2(\Lambda)\wp(z)-g_3(\Lambda).\]
The equation $Y^2 = 4X^3  - g_2(\Lambda)X-g_3(\Lambda)$ defines an elliptic function field $\mathbb{C}(x,y)\cong \mathbb{C}(\wp,\wp')$. The Uniformization Theorem constructs such an isomorphism that maps $x\mapsto\wp$ and $y\mapsto \wp'$, so $\fieldAlg = \mathbb{C}(x,y)\cong \fieldC$.

The Uniformization Theorem lets us glean lots of intuition from Chapter~\ref{c:analytic}. Previously we utilized an analytic construct $\zeta_{\Lambda}(z)$ on the universal covering space $\mathbb{C}$ of the torus $\mathbb{C}/\Lambda$ to compute reductions and find criteria for summability. The Uniformization Theorem says that one should in principle be able to see this entirely from the algebraic perspective, and this is our goal for this chapter.

Of course, the Uniformiztion Theorem only applies for elliptic curves defined over $\mathbb{C}$. The field $\mathbb{C}$ is a highly analytic construct, and the case of general fields $K$ has no associated analytic objects a priori. This complicates the use of explicit constructions like the Weierstrass $\zeta$-function. In the following subsection we use strategies to invoke the existence of functions to use in the algebraic analogue of the reduction from Subsection~\ref{ss:reductionSearch}.

\subsection{Divisors and the Riemann--Roch Theorem}

In order to talk about a field $\fieldAlg$ of meromorphic functions on an elliptic curve, we must import language from algebraic geometry in order to express the data of elements of $\fieldAlg$. The strategy of Subsection~\ref{ss:analyticReduction} was to use elliptic functions with very particular poles to reduce $f\in\fieldC$ to have poles with minimal dispersion with respect to the shift automorphism. In the algebraic setting we use divisors of elliptic functions to express the data of their poles. But to show that elliptic functions with particular divisors even exist, we must employ the Riemann--Roch theorem. We review these technologies here. Definitions and results can mostly be found in \cite[Chapter~II]{Silverman2009}. Results outside this reference are cited explicitly.

The data of a set of poles of a function is simply an unordered list with multiplicities, which can be modeled with formal sums.

\begin{definition}
Given a projective curve $C$, its \defemph{divisor group} $\Div(C)$ is the free abelian group generated by the points on $C$. It consists of formal sums
\[\sum\limits_{P\in C} n_P[P]\]
of finite support, where $n_P\in\mathbb{Z}$ for all $P\in C$. The \defemph{degree} of a \defemph{divisor} $D\in\Div(C)$ with the above expression is defined to be
\[\deg(D) = \sum\limits_{P\in C} n_P.\]
The subgroup of divisors of degree $0$ is denoted by $\Div^0(C)$.
\end{definition}

\noindent Divisors are a useful general construct for storing the data of poles \emph{and} zeros of a function on $C$.

\begin{definition}
Suppose a curve $C$ is defined over the field $K$, and $f$ is a nonzero meromorphic function on $C$. Then we define the divisor associated to $f$ by
\[\div(f) = \sum\limits_{P\in C} n_P[P],\]
where $n_P=m$ if $P$ is a zero of $f$ of multiplicity $m$, $n_P=-m$ if $P$ is a pole of $f$ of multiplicity $m$, and $n_P=0$ otherwise.
\end{definition}

What is more, if $C$ is smooth then the zeros and poles of a meromorphic function on $f$ need to play a balancing game similar to Theorem~\ref{thm:balancingGame}, even if it is less powerful in the case of genus different from $1$.

\begin{proposition}
    \label{prop:algDivisorsOfFunctionsVanish}
    Let $C$ be a smooth curve, and let $f$ be a nonzero meromorphic function on $C$. Then $\deg(\div(f))=0$. The divisor $\div(f)$ is identically $0$ if and only if $f$ is constant.
\end{proposition}

In Chapter~\ref{c:analytic} we were able to invoke explicit constructions of elliptic functions with particular poles, but in the algebraic setting we no longer have such recipes for functions in $\fieldAlg$. To circumvent this we observe that if we were to limit the set of ``permissible'' poles of a function to a finite set, then this limits the choice of functions to a $K$-vector space. Let us provide notation for this concept.

First, we say that a divisor $D$ is \defemph{effective} if it has an expression
\[D=\sum\limits_{P\in C} n_P[P]\]
where the coefficients $n_P$ are in fact all nonnegative. This creates a partial ordering where $D_1\geq D_2$ if and only if $D_1-D_2$ is effective.

\begin{definition}
    Let $C$ be a projective curve, and let $K(C)$ be its field of meromorphic functions. Let $D\in\Div(C)$. We associate with $D$ the \defemph{Riemann--Roch space}
    \[\mathcal{L}(D) = \{f\in K(C)\mid \div(f)\geq -D\}.\]
    It can be shown that this set is a vector space over $K$: denote its dimension by $\ell(D)$.\footnote{This definition is formulated in \cite[Section~II.6]{hartshorneAlgebraicGeometry1977} in terms of invertible sheaves. The sheaf-theoretic view is unnecessary for our purposes.}
\end{definition} \medskip

\begin{example}
    For points $P,Q\in C$, the vector space $\mathcal{L}([P]+[Q])$ is that of functions only allowed to have poles of order at most $1$ and $P$ and $Q$. To contrast, $\mathcal{L}([P]-[Q])$ is the vector space of functions only allowed to have a pole of order $1$ at $P$, but \emph{required} to have a zero of order at least $1$ at $Q$.
\end{example}

The dimension $\ell(D)$ of a Riemann--Roch space $\mathcal{L}(D)$, at least on a \emph{smooth} curve $C$, is precisely the concern of the Riemann--Roch Theorem. The utility of such a result is that if we know the dimension of such a vector space, we can infer the existence of functions on a smooth curve with a prescribed set of poles. The exact statement of the Riemann--Roch Theorem requires more technology than we have already developed: for example see \cite[Theorem~II.5.4]{Silverman2009} or \cite[Theorem~6.6.37]{Kirwan1992}. In the interest of time and space we formulate a version here that that removes the differential terminology while being suitable for our purposes.

\begin{theorem}[Riemann--Roch Theorem, genus $1$]
\label{thm:riemannroch} Let $E$ be a nonsingular curve of genus $1$. Then $\ell(D) = \deg(D)$ for all nonzero effective divisors $D\in\Div(E)$.
\end{theorem} \medskip

\begin{proof}
    See \cite[Example~II.5.7]{Silverman2009}.
\end{proof}

The Riemann--Roch Theorem is an incredibly powerful tool for invoking the existence of functions with particular divisors without needing to have a construction on hand. Some constructions are immediate.

\begin{example}
\label{ex:RRapplications}
    There are several basic examples of applications of Theorem~\ref{thm:riemannroch} that are instructive and inform our later approaches to determing summability in the algebraic setting.
    \begin{itemize}
        \item Suppose $D$ is an effective divisor with degree $1$: then $D=[P]$ for some $P\in E$. From Theorem~\ref{thm:riemannroch}, $\ell(D)=1$. Because the constant functions are all elements of $\mathcal{L}(D)$ it follows that they comprise $\mathcal{L}(D)$. An immediate corollary is that all nonconstant elliptic functions have at least two poles counting multiplicity, mirroring Corollary~\ref{cor:2poles}.
        \item If $D=[P]+[Q]$ for some $P,Q\in E$ not necessarily distinct, then $\ell(D)=2$. Again $K\subset\mathcal{L}(D)$ realized as constant functions on $E$, so there must exist some $f\in \mathcal{L}(D)$ with $\div(f) = D$.
        \item If $D=n[P]$ with $n\geq 2$ then one can inductively show that $\mathcal{L}(D)$ has a basis $\{1,f_2,f_3,\dots,f_n\}$ where $1\in K$ and $\div(f_k)=k[P]$ for $2\leq k\leq n$.
        \item If $D=n[P]+[Q]$ with $n\geq 1$ then one can inductively show that $\mathcal{L}(D)$ has a basis $\{1,f_1,f_2,\dots,f_n\}$ where $1\in K$ and $\div(f_k)=k[P]+[Q]$ for $1\leq k\leq n$.
    \end{itemize}
\end{example} \medskip\medskip

With our newfound technology of divisors we can see how the reduction steps of Subsection~\ref{ss:analyticReduction} can be expressed using this language.

\begin{example}
    In the reduction of Subsection~\ref{ss:analyticReduction}, in terms of divisors on $\mathbb{C}/\Lambda$,
    \begin{itemize}
        \item Step $1$ used functions of the form
        \[f_1(z) = \frac{d^{j-1}}{dz^{j-1}} \zeta(z-\hat{q}-n\shiftnumber),\]
        each of which has divisor $j\Big[\overline{\hat{q}+n\shiftnumber}\Big]$ where the overline denotes taking modulo $\Lambda$.
        \item Step $2$ used functions of the form
        \[f_2(z) = \zeta(z-\hat{q}-n\shiftnumber) - \zeta(z-n\shiftnumber),\]
        each of which has divisor $\Big[\overline{\hat{q}+n\shiftnumber}\Big] + \big[\overline{n\shiftnumber}\big]$. This step is notable because there is exactly one function with this divisor up to a constant multiple and additional constant term, as can be seen from Example~\ref{ex:RRapplications}. Indeed,
        \[\mathcal{L}\Big(\Big[\overline{\hat{q}+n\shiftnumber}\Big] + \big[\overline{n\shiftnumber}\big]\Big) = \Span_{\mathbb{C}}\left\{ 1,\zeta(z-\hat{q}-n\shiftnumber) - \zeta(z-n\shiftnumber)\right\}.\]
        Similar structures can be seen in step $3$ of the reduction.        
    \end{itemize}
\end{example}

\section{Previous Results on Summability}

We now have a sufficient collection of algebraic geometric technology to handle the question of summability. We are also now able to precisely state the previous results of \cite{Dreyfus2018} on the matter. As we will see, these results will mirror our previous work in the analytic setting of Chapter~\ref{c:analytic}.

An additional difficulty in the algebraic setting is that we no longer have access to a partial fraction decomposition, or even a global expression that simulates such a decomposition. Recall that such a decomposition is an amalgamation of all local expansions of an elliptic function. One approach to circumvent finding a global expression is to simply keep track of all local expansions separately. This is precisely the implementation that the authors of \cite{Dreyfus2018} use to define orbital residues in the algebraic setting.

After defining the difference field structure on $\fieldAlg$, we summarize the previous results on summability. We also tie these results back to results in Chapter~\ref{c:analytic}.

\subsection{The Shift Automorphism}

Given an elliptic curve $E$, let $\fieldAlg$ be the field of elliptic functions on $E$. We will define shift automorphism $\tau$ on $\fieldAlg$ in terms of the group law on $E$. Let $\shiftalg$ be a non-torsion point of $E$, that is, let $\shiftalg\in E$ be such that $n\shiftalg\neq O$ for all $n\neq 0$. Given $f(Q)\in\fieldAlg$ with variable $Q$, we define the shift automorphism $\tau$ on $\fieldAlg$ by
\[\tau:f(Q)\mapsto f(Q\oplus \shiftalg).\]
We denote the corresponding shift on the points of $E$ by
\begin{align*}
    \shifttorus&:E\to E\\
    \shifttorus&:Q\mapsto Q\ominus \shiftalg,
\end{align*}
where $\shifttorus$ is a morphism of sets as opposed to groups. Thus $\shift^{-1}(f(Q))=f(\shifttorus(Q))$ for all $f\in\fieldAlg$. So just as before, a point $Q\in E$ is a pole of order $k$ of $f\in\fieldAlg$ if and only if $\shifttorus(Q)$ is a pole of order $k$ of $\shift(f)\in\fieldAlg$. The definition of summability is as before: an element $f\in \fieldAlg$ is called \defemph{summable} if there exists $g\in \fieldAlg$ such that $f = \shift(g)-g$. The polar dispersion $\pdisp(f)$ and the weak polar dispersion $\wpdisp(f)$ have analogous definitions and satisfy analogous properties.

With these concepts we can show an algebraic analogue of Lemma~\ref{lem:analyticLPSummable}:

\begin{lemma}
    \label{lem:algebraicLPSummable}
    Let $Q$ be a point on the elliptic curve $E$, and let
    \[f\in \mathcal{L}([Q]+[\shifttorus(Q)])\]
    be nonzero. Then $f$ is not summable.
\end{lemma}

The proof is analogous to the analytic case: analyzing the polar dispersion of a hypothetical $g\in \fieldAlg$ satisfying $\shift(g)-g=f$ shows that $g$ has at most one simple pole, which is impossible from the discussion of Example~\ref{ex:RRapplications}. We will relate this result to orbital residues in the following discussion.

\subsection{Orbital Residues}

Orbital residues of elliptic functions were first developed in \cite[Appendix~B]{Dreyfus2018} in precisely this algebraic setting. The authors circumvented the issue of no global decompositions by simply writing the orbital residues in terms of the local expansions. These expansions are in terms of \defemph{local uniformizers} $u_Q$ at each point $Q\in E$, where $u_Q$ is an elliptic function with a zero at $Q$ of multiplicity $1$ (cf. \cite[p.~137]{hartshorneAlgebraicGeometry1977}). But in order to aggregate the data along the orbitals of the shift action $\shifttorus$, one must ensure some compatability of the local uniformizers along this orbit. Dreyfus et al. accomplish this with \cite[Definition~B.6]{Dreyfus2018}:

\begin{definition}
    Let $\mathcal{U}=\{u_Q\mid Q\in E\}$ be a set of local uniformizers at the points of $E$. We say $\mathcal{U}$ is a \defemph{coherent set of local uniformizers} if for any $Q\in E$,
    \[u_{\shifttorus(Q)} = \shift(u_Q).\]
\end{definition}

This is \emph{not} to say that $u_{Q_1\ominus Q_2}(Q_3) = u_{Q_1}(Q_3\oplus Q_2)$ for all $Q_1,Q_2,Q_3\in E$: the uniformizers are not required to respect general shifts. The uniformizers chosen on different orbits of $\shifttorus$ do not interact with each other, so coherent sets of local uniformizers can be created from one another by altering those uniformizers on a single orbit of $\shifttorus$.

\begin{definition}
    Let $\mathcal{U}$ be a set of local uniformizers at the points of $E$. If $f\in\fieldAlg$ has a pole at $Q$ of order $n$, we may write
    \[f = \frac{c_{Q,n}}{u_Q^n} + \cdots +\frac{c_{Q,2}}{u_Q^2} + \frac{c_{Q,1}}{u_Q} + g\]
    where $g$ is elliptic and holomorphic at $Q$, and $c_{Q,j}\in K$ for all $j\in\mathbb{N}$. We shall refer to $c_{Q,j}$ as the \defemph{algebraic residue} of $f$ of order $j$ at $Q$, and denote it by $\res_{\mathcal{U}}(f,Q,j)$.
\end{definition}

\begin{remark}
    To reiterate previous remarks, the terminology overloads the term ``residue'', which are attributes of differentials. The terminology here is chosen to remain congruous with the previous work in \cite{Chen2012} and \cite{Dreyfus2018}.
\end{remark}

The aggregation of these algebraic residues via the orbits of $\shifttorus$ forms the notion of orbital residues, as in \cite[Definition~B.7]{Dreyfus2018}.

\begin{definition}\label{def:ores}
    Let $f$ be an elliptic function and $\mathcal{U}=\{u_Q\mid Q\in E\}$ be a coherent set of local uniformizers. Denote the bidirectional orbit of $Q$ through $\shifttorus$ by
    \[ \omega_Q := \{ Q\oplus i\shiftalg\mid i\in\mathbb{Z}\} = \{ (\shifttorus)^{i}(Q)\mid i\in\mathbb{Z}\}.\]
    For each $j\in \mathbb{N}_{>0}$ we define the \defemph{algebraic orbital residue} of $f$ of order $j$ at the orbit $\omega$ to be
    \[\ores_{\mathcal{U}}(f,\omega,j) = \sum\limits_{Q\in \omega} \res_{\mathcal{U}}(f,Q,j).\]
\end{definition}\medskip\medskip

In contrast to Chapter~\ref{c:analytic} we retain the subscript $\mathcal{U}$ in the notation to indicate that the orbital residue is highly dependent on the choice of the set $\mathcal{U}$ of local uniformizers.

Notice that $\ores_{\mathcal{U}}(f,\omega_Q,r)$ vanishes if there do not exist poles in the orbit $\omega_Q$ that have order at least $r$. These residues can then be used to formulate, up to a term with two or fewer simple poles, a criterion for elliptic functions, shown in \cite[Proposition~B.8]{Dreyfus2018}.

\begin{proposition}\label{prop:algebraicAlmostSummable}
    Let $f\in \fieldAlg$ and $\mathcal{U}=\{u_Q\mid Q\in E\}$ be a coherent set of local uniformizers. The following are equivalent.
    \begin{enumerate}[label=(\arabic*)]
        \item There exist $Q\in E$, $g\in \fieldAlg$, and $e\in\mathcal{L}([Q]+[\shifttorus(Q)])$ such that $f=\shift(g)-g+e$.
        \item For any $Q\in E$ and $j\in \mathbb{N}_{>0}$ we have $\ores_{\mathcal{U}}(f,\omega_Q,j)=0$.
    \end{enumerate}
\end{proposition} \medskip

That is to say, for any elliptic function $f$ for which its orbital residues vanish, there exists a reduction by a summable function $\shift(g)-g$ to a function $e$ with at most two simple poles related to each other by $\shifttorus$. We can see this result when we look closely at the reduction in the analytic case in Theorem~\ref{thm:analyticReduction}: when the analytic orbital residues vanish, all terms in the reduction disappear besides those including the analytic panorbital residues, namely a term in $\mathcal{L}([0]+[\shiftnumber])$. In fact, our derivation of the terms Theorem~\ref{thm:analyticReduction} that \emph{don't} involve the panorbital residues mimics the proof of \cite[Proposition~B.8]{Dreyfus2018} while giving an explicit reduction process.

This result shows that the problem of summability reduces to determining where $g\in\mathcal{L}([Q]+[\shifttorus(Q)])$ is summable. From Lemma~\ref{lem:algebraicLPSummable} this occurs precisely when $g$ vanishes. It would be computationally convenient of the orbital residue could detect this case. Unfortunately, the worst-case scenario occurs:

\begin{lemma}
    If $g\in\mathcal{L}([Q]+[\shifttorus(Q)])$ then $\ores_{\mathcal{U}}(g,\omega_Q,r)=0$ for all $r\geq 1$.
\end{lemma}

Because $g$ has at most simple poles the result is clear for $r>1$. The result when $r=1$ is a special case of \cite[Lemma~B.15]{Dreyfus2018} in which the function has two simple poles in the $\shifttorus$-orbit of $Q$. It is a technical lemma that requires the fact that the residues of a differential form on a compact Riemann surface must sum to zero. In the interest of space, we merely cite the proof and do not replicate it here.

Proposition~\ref{prop:algebraicAlmostSummable} and Lemma~\ref{lem:algebraicLPSummable} show that algebraic orbital residues fail as an oracle of determining summability, just as in the analytic case. Our goal is the same in the algebraic setting: define panorbital residues that let us define a reduction of elliptic functions similar to Theorem~\ref{thm:analyticReduction} that lets us see summability from knowing the values of a set of constants.

\section{Reduction of Elliptic Functions}

Just as in Chapter~\ref{c:analytic}, a reduction of an elliptic function $f\in\fieldAlg$ will be highly nonunique. In the analytic setting we used an analytic pinning to make a choice of to where the poles of $f$ will be gathered in the reduction. Here we must make even more choices. 

\begin{definition}
An \defemph{algebraic pinning} $\Xi$ is a collection of the following choices:
\begin{itemize}
    \item A coherent set $\mathcal{U}_{\Xi}$ of local uniformizers $u_Q$ for each point $Q$ on the curve,
    \item A set $\mathcal{R}_{\Xi}$ of representatives of $\shifttorus$ orbits of $E$.
    \item An anchor $\widehat{Q}\in\mathcal{R}_{\Xi}$.
\end{itemize}
\end{definition}

In all following discussions where $\Xi$ is considered to be fixed, we make the convention of dropping the subscript $\Xi$ from $\mathcal{U}_{\Xi}$ and $\mathcal{R}_{\Xi}$.

We wish to reduce an elliptic function $f$, by way of subtracting summable elements, to an elliptic function $\overline{f}$ with poles only at $\widehat{Q}$, $\widehat{Q}\oplus \shiftalg$, and the orbit representatives in $\mathcal{R}$. Depending on the choices made above the clerical work involved will vary in difficulty. The easiest case is when the poles of $f$ are all to one side of all the elements of $\mathcal{R}$. Denoting the orbit of $\widehat{Q}$ by $\widehat{\omega}=\{\widehat{Q}\oplus n\shiftalg\mid n\in\mathbb{Z}\}$, we also wish to specify that $\widehat{\omega}$ is free of any poles of $f$: this is to ensure that any poles that ``rain down'' from other orbits will not be interacting with any preexisting poles on $\widehat{\omega}$. With this in mind we make the following definition.

\begin{definition}
For an elliptic function $f$, let $\sing(f)$ denote the set of poles of $f$. We say an algebraic pinning $\Xi$ is \defemph{admissible} if the following are true:
\begin{align*}
    \sing(f) \cap \widehat{\omega} &= \emptyset,\\
    \sing(f)\cap \{Q\ominus n\shiftalg \mid n\in\mathbb{Z}_{\geq 0}\} &= \emptyset \ \text{ for all }Q\in\mathcal{R}_{\Xi}.
\end{align*}
\end{definition}

We must now define the functions with which we reduce. These functions will be defined in terms of their divisors and principal parts that we may utilize in the reduction. First, fix an admissible pinning $\Xi$. For each $\shifttorus$-orbit $\omega$, denote by $Q_{\omega}$ the unique point in the intersection $\omega\cap \mathcal{R}$.

\begin{lemma}
For each $\widehat{\omega}\neq \omega\in \Omega$ and $j\in \mathbb{N}$, there exist unique \[\varphi_{\omega,j}\in\mathcal{L}\left(j[Q_\omega]+\Bigl[\widehat{Q}\Bigr]\right)\qquad\text{and}\qquad d_j(\omega)\in K\qquad \text{such that}\]
\begin{itemize} 
\item $\varphi_{\omega,j}-u_{Q_\omega}^{-j}$ is nonsingular at $Q_\omega$; and \item $\varphi_{\omega,j}-d_j(\omega)\cdot u_{\widehat{Q}}^{-1}$ has a zero at $\widehat{Q}$.
\end{itemize}
\end{lemma} \medskip

\begin{proof}
    We proceed by strong induction on $j$. Consider any choice of $\varphi_{\omega,1}\in\mathcal{L}\left([Q_\omega]+[\widehat{Q}]\right)$ and scale by a constant in $K$ so that $\varphi_{\omega,1}-u_{Q_{\omega}}^{-1}$ is nonsingular at $Q_{\omega}$. There clearly exists $d_1(\omega)$ for which $\varphi_{\omega,1}-d_1(\omega)\cdot u_{\widehat{Q}}^{-1}$ is nonsingular at $\widehat{Q}$. Add the proper constant in $K$ so it in fact has a zero at $\widehat{Q}$. This proves the base case $j=1$.

    For the inductive step assume that for all $j<M$ the result holds, for some $M\in\mathbb{Z}_{\geq 2}$. A quick inductive proof shows the existence of an element $f$ in the Riemann--Roch space with divisor $j[Q_{\omega}]+[\widehat{Q}]$ from Theorem~\ref{thm:riemannroch}. Suppose the local expansion of $f$ at $Q_{\omega}$ is
    \[f = \frac{c_j}{u_{Q_{\omega}}^j} + \cdots + \frac{c_1}{u_{Q_{\omega}}} + g,\]
    where $g$ is holomorphic at $Q_{\omega}$ and $c_1,\cdots,c_j\in K$ and $c_j\neq 0$. Then
    \[\frac{1}{c_j}f - \sum\limits_{k=1}^{j-1} \frac{c_k}{c_j}\varphi_{\omega,k}\]
    is nonsingular at $Q_{\omega}$ from the inductive hypothesis. Denote this function by $\varphi_{\omega,j}$. Then there exists $d_j(\omega)$ such that $\varphi_{\omega,j}-d_j(\omega)\cdot u_{\widehat{Q}}^{-1}$ is nonsingular at $\widehat{Q}$. Again correct by adding a constant in $K$ to ensure that it in fact has a zero at $\widehat{Q}$. This proves the inductive step, which completes the proof of the lemma.
\end{proof}

\begin{lemma}
    For each $j\in\mathbb{Z}_{\geq 2}$ there exists a unique
\[\psi_j\in\mathcal{L}\left( \Bigl[\widehat{Q}\oplus (j-1)\shiftalg\Bigr]+\Bigl[\widehat{Q}\Bigr]    \right)\]
such that $\psi_j+u_{\widehat{Q}}^{-1}$ has a zero at $\widehat{Q}$.
\end{lemma}
\begin{proof}
    Consider any choice of nonsingular $\psi_j$ and scale by a constant in $K$ so that $\psi_j+u_{\widehat{Q}}^{-1}$ is nonsingular at $\widehat{Q}$. Add the proper constant in $K$ so it in fact has a zero at $\widehat{Q}$.
\end{proof}

Using these functions we reduce $f$ in two steps.
\begin{enumerate}[label=\arabic*.]
    \item Partially reduce $f$ to $\tilde{f}$ such that for all orbits $\omega\neq\widehat{\omega}$ the only pole of $f$ on $\omega$ is $Q_{\omega}$, creating order $1$ poles of $\tilde{f}$ on $\widehat{\omega}$ in the process.
    \item Reduce $\tilde{f}$ to $\overline{f}$ by collecting all poles on $\widehat{\omega}$ to $\widehat{Q}$ and $\widehat{Q}\oplus \shiftalg$.
\end{enumerate}

Throughout the following computations, we will primarily use the difference operator $\Delta^{(n)} := \Id-\shift^{-n}$ for all $n\in\mathbb{N}$. That is,
\[\Delta^{(n)}(f) = f-\shift^{-n}(f) = \sum\limits_{k=-n}^{1}\shift(\shift^k(f))-\shift^k(f).\]
Because $\shift$ is an automorphism on $\fieldAlg$ it follows that $\Delta^{(n)}(f)$ is summable for all $f\in\fieldAlg$ and $n\in\mathbb{N}$. Notice furthermore that $\Delta^{(n)}$ is \emph{not} defined to be the $n$-fold composition of $\Delta^{(1)}$.

\subsection{The First Reduction}

\begin{proposition}\label{prop:algReduction1}
    Let $\Xi$ be an admissible algebraic pinning of $f$. The following reduction $\tilde{f}$ of $f$ only has order-$1$ poles on $\widehat{\omega}$, and all other poles are elements of $\mathcal{R}$.
    
    \[\tilde{f} = f+\sum\limits_{\omega\neq \widehat{\omega}}\sum\limits_{m\geq 1}\sum\limits_{j\geq 1} \res_{\mathcal{U}}(f,Q_{\omega}\oplus m\shiftalg,j)\cdot \Delta^{(m)}(\varphi_{\omega,j}).\]
\end{proposition}

\begin{proof}
    First confirm that the above is indeed a reduction: $\Delta^{(m)}(\varphi_{\omega,j})$ is summable, so $f-\tilde{f}$ is a linear combination of summable functions and therefore summable.

    Now let us consider the local expansions of $\Delta^{(m)}(\varphi_{\omega,j})$. By definition the nonzero residues of $\varphi_{\omega,j}$ are
    \[\res_{\mathcal{U}}(\varphi_{\omega,j},Q_{\omega},j) = 1,\quad \res_{\mathcal{U}}(\varphi_{\omega,j},\widehat{Q},1) = d_j(\omega).\]
    So the nonzero residues of $\Delta^{(m)}(\varphi_{\omega,j})$ are
    \begin{align*}
        &\res_{\mathcal{U}}(\Delta^{(m)}(\varphi_{\omega,j}),Q_{\omega},j) = 1,\quad \res_{\mathcal{U}}(\Delta^{(m)}(\varphi_{\omega,j}),Q_{\omega}\oplus n\shiftalg,j) = -1,\\
        &\res_{\mathcal{U}}(\Delta^{(m)}(\varphi_{\omega,j}),\widehat{Q},1) = d_j(\omega),\quad \res_{\mathcal{U}}(\Delta^{(m)}(\varphi_{\omega,j}),\widehat{Q}\oplus n\shiftalg,1) = -d_j(\omega).
    \end{align*}
    Therefore the only terms that affect the residue of $\tilde{f}-f$ at $Q_{\omega}\oplus m\shiftalg$ of order $j$ for fixed $\omega\neq\widehat{\omega}$ and fixed $m,j\in\mathbb{N}$ are $f$ itself and the $\Delta^{(m)}(\varphi_{\omega,j})$ term. Therefore, for $\omega\neq\widehat{\omega}$,
    \[\res_{\mathcal{U}}(\tilde{f},Q_{\omega}\oplus m\shiftalg,j) = \res_{\mathcal{U}}(f,Q_{\omega}\oplus m\shiftalg,j) + \res_{\mathcal{U}}(f,Q_{\omega}\oplus m\shiftalg,j)\cdot (-1) = 0.\]
    The terms $\Delta^{(m)}(\varphi_{\omega,j})$ have no poles at $Q_{\omega}\ominus m\shiftalg$ for any $m\in\mathbb{N}$, so all residues of $\tilde{f}$ on orbits $\omega\neq\widehat{\omega}$ vanish except at the elements of $\mathcal{R}$. Because $f$ has no poles on $\widehat{\omega}$ and $\varphi_{\omega,j}$ only has poles of order $1$ on $\widehat{\omega}$, this proves the result.
    \end{proof}

To collect our bearings we now recompute all other residues of $\tilde{f}$ with respect to the algebraic pinning after this first reduction.

\begin{lemma}
    Let $\Xi$ be an admissible algebraic pinning of $f$. For $\tilde{f}$ as defined in Proposition~\ref{prop:algReduction1}, the remaining residues of $\tilde{f}$ are as follows:

    \begin{itemize}
        \item For $\omega\neq\widehat{\omega}$,
        \[\res_{\mathcal{U}}(\tilde{f},Q_{\omega},j) = \ores_{\mathcal{U}}(f,\omega,j).\]
        \item For $n>0$,
        \[\res_{\mathcal{U}}(\tilde{f},\widehat{Q}\oplus n\shiftalg,1) = -\sum\limits_{\omega\neq\widehat{\omega}}\sum\limits_{j\geq 1}\res_{\mathcal{U}} (f,Q_{\omega}\oplus n\shiftalg,j)\cdot d_j(\omega).\]
        \item For $n=0$,
        \[\res_{\mathcal{U}}(\tilde{f},\widehat{Q},1) = \sum\limits_{\omega\neq\widehat{\omega}}\sum\limits_{m\geq 1}\sum\limits_{j\geq 1}\res_{\mathcal{U}} (f,Q_{\omega}\oplus m\shiftalg,j)\cdot d_j(\omega).\]
    \end{itemize}
    All other residues of $\tilde{f}$ with respect to the algebraic pinning are zero.
\end{lemma}\medskip

\begin{proof}
    For $\omega\neq\widehat{\omega}$, as stated in the proof of Proposition~\ref{prop:algReduction1},
    \[\res_{\mathcal{U}}(\Delta^{(m)}(\varphi_{\omega,j}),Q_{\omega},j) = 1.\]
    Therefore
    \[\res_{\mathcal{U}}(\tilde{f},Q_{\omega},j) = \res_{\mathcal{U}}(f,Q_{\omega},j) + \sum\limits_{m\geq 1} \res_{\mathcal{U}}(f,Q_{\omega}\oplus n\shiftalg,j)\cdot 1 .\]
    Because we chose an admissible algebraic pinning, $\res_{\mathcal{U}}(f,Q_{\omega},j)=0$. Similarly, the sum is precisely the orbital residue $\ores_{\mathcal{U}}(f,\omega,j)$. This proves that $\res_{\mathcal{U}}(\tilde{f},Q_{\omega},j)=\ores_{\mathcal{U}}(f,\omega,j)$ for all $\omega\neq\widehat{\omega}$.

    For $n>0$, note that $\res_{\mathcal{U}}(f,\widehat{Q}\oplus n\shiftalg,1)=0$ due to the pinning being admissible, so only $\Delta^{(n)}(\varphi_{\omega,j})$ will contribute to $\res_{\mathcal{U}}(\tilde{f},\widehat{Q}\oplus n\shiftalg,1)$ for each orbit $\omega\neq\widehat{\omega}$ and order $j\geq 1$, and as in the proof Proposition~\ref{prop:algReduction1},
    \[\res_{\mathcal{U}}(\Delta^{(n)}(\varphi_{\omega,j}),\widehat{Q}\oplus n\shiftalg,1) = -d_j(\omega).\]
    
    Therefore
    \begin{align*}
        \res_{\mathcal{U}}(\tilde{f},\widehat{Q}\oplus n\shiftalg,1)  &= \sum\limits_{\omega\neq\widehat{\omega}}\sum\limits_{j\geq 1} \res_{\mathcal{U}}(f,Q_{\omega}\oplus n\shiftalg,j)\cdot (-d_j(\omega))\\
        &=-\sum\limits_{\omega\neq\widehat{\omega}}\sum\limits_{j\geq 1} \res_{\mathcal{U}}(f,Q_{\omega}\oplus n\shiftalg,j)\cdot d_j(\omega).
    \end{align*}

    For the final residue $\res_{\mathcal{U}}(\tilde{f},\widehat{Q},1)$, note that $\res_{\mathcal{U}}(f,\widehat{Q},1) = 0$ due to $f$ being admissible, so only the terms $\Delta^{(m)}(\varphi_{\omega,j})$ will contribute to this residue for all $\omega\neq\widehat{\omega}$, all $m\geq 1$, and all orders $j\geq 1$, and
    \[\res_{\mathcal{U}} (\Delta^{(m)}(\varphi_{\omega,j}), \widehat{Q},1)  = d_j(\omega).\]
    Therefore
    \[\res_{\mathcal{U}}(\tilde{f},\widehat{Q},1)  = \sum\limits_{\omega\neq\widehat{\omega}}\sum\limits_{m\geq 1}\sum\limits_{j\geq 1} \res_{\mathcal{U}}(f,Q_{\omega}\oplus m\shiftalg,j)\cdot d_j(\omega).\]

    It remains to confirm that all uncomputed residues of $\tilde{f}$ actually vanish. From Proposition~\ref{prop:algReduction1} the only residues we have left to analyze are $\res_{\mathcal{U}}(\tilde{f},\widehat{Q}\oplus n\shiftalg,1)$ for $n<0$. All functions involved in the reduction, including $f$ itself, are by definition regular at $\widehat{Q}\oplus n\shiftalg$ for $n<0$. This proves the lemma.
\end{proof}

\newpage

\subsection{The Second Reduction}

\begin{proposition}\label{prop:algReduction2}
    Let $\Xi$ be an admissible algebraic pinning of $f$. Let $\tilde{f}$ be defined as in Proposition~\ref{prop:algReduction1}. The following reduction $\overline{f}$ of $f$ only has order $1$ poles on $\widehat{\omega}$ at $\hat{Q}$ and $\hat{Q}\oplus \shiftalg$ in particular, and all other poles are elements of $\mathcal{R}$.
    
    \[\overline{f} = \tilde{f} - \sum\limits_{m\geq 2} \sum\limits_{k=2}^{m} \left[\Delta^{(1)}(\psi_k) \sum\limits_{\omega\neq\widehat{\omega}} \sum\limits_{j\geq 1} \res_{\mathcal{U}}(f,Q_{\omega}\oplus m\shiftalg,j)\cdot d_j(\omega)\right].\]
\end{proposition}

\begin{proof}
    Each $\psi_k$ has its two poles on $\widehat{\omega}$, so all poles of $\overline{f}$ off of $\widehat{\omega}$ are elements of $\mathcal{R}$ by virtue of Proposition~\ref{prop:algReduction1}. Furthermore, $\tilde{f}$ and $\psi_k$ only have order $1$ poles on $\widehat{\omega}$, so $\overline{f}$ does as well. In fact all these poles are at $\widehat{Q}\oplus n\shiftalg$ for $n\geq 0$. It therefore only remains to verify that
    \[\res_{\mathcal{U}}(\overline{f},\widehat{Q}\oplus n\shiftalg,1)=0\]
    for all integers $n\geq 2$.

    By definition, the poles of
    \[\Delta^{(1)}(\psi_k) = \psi_k-\tau^{-1}(\psi_k)\]
    are at $\widehat{Q}$ and $\widehat{Q}\oplus (k-1)\shiftalg$ contributed by $\psi_k$, and at $\widehat{Q}\oplus\shiftalg$ and $\widehat{Q}\oplus k\shiftalg$ contributed by $\tau^{-1}(\psi_k)$. In particular,
    \[\res_{\mathcal{U}}(\Delta^{(1)}(\psi_k),\widehat{Q}\oplus kS,1) = -\res_{\mathcal{U}}(\shift^{-1}(\psi_k),\widehat{Q}\oplus k\shiftalg,1) = -1.\]
    Therefore the order $1$ residue of $\overline{f}$ at $\widehat{Q}\oplus n\shiftalg$ for $n\geq 2$ is
    \begin{multline}
    \res_{\mathcal{U}}(\overline{f},\widehat{Q}\oplus n\shiftalg,1)
    = \  \res_{\mathcal{U}}(\tilde{f},\widehat{Q}\oplus n\shiftalg,1) \\
    - \sum\limits_{m\geq n+1} \sum_{k=n}^{n+1} \res_{\mathcal{U}}(\Delta^{(1)}(\psi_k),Q_{\omega}\oplus n\shiftalg,1)\cdot \sum\limits_{\omega\neq\widehat{\omega}}\sum\limits_{j\geq 1} \res_{\mathcal{U}}(f,Q_{\omega}\oplus m\shiftalg,j)\cdot d_j(\omega)\\
    - \res_{\mathcal{U}}(\Delta^{(1)}(\psi_n),Q_{\omega}\oplus n\shiftalg,1)\cdot\sum\limits_{\omega\neq\widehat{\omega}}\sum\limits_{j\geq 1}\res_{\mathcal{U}}(f,Q_{\omega}\oplus n\shiftalg,j)\cdot d_j(\omega)\nonumber\end{multline}

    \begin{align*}
    = \ & \res_{\mathcal{U}}(\tilde{f},\widehat{Q}\oplus n\shiftalg,1) 
     - \sum\limits_{m\geq n+1} (1-1)\cdot \sum\limits_{\omega\neq\widehat{\omega}}\sum\limits_{j\geq 1} \res_{\mathcal{U}}(f,Q_{\omega}\oplus m\shiftalg,j)\cdot d_j(\omega)\\
    & +\sum\limits_{\omega\neq\widehat{\omega}}\sum\limits_{j\geq 1}\res_{\mathcal{U}}(f,Q_{\omega}\oplus n\shiftalg,j)\cdot d_j(\omega)\\
    = \ & -\sum\limits_{\omega\neq\widehat{\omega}}\sum\limits_{j\geq 1}\res_{\mathcal{U}} (f,Q_{\omega}\oplus n\shiftalg,j)\cdot d_j(\omega) + \sum\limits_{\omega\neq\widehat{\omega}}\sum\limits_{j\geq 1}\res_{\mathcal{U}}(f,Q_{\omega}\oplus n\shiftalg,j)\cdot d_j(\omega) \\
    = \ & 0. 
    \end{align*}
    This proves that the above reduction has its poles at the prescribed locations.
\end{proof}

\begin{proposition}
\label{prop:algReductionEnd}
    Let $\psi_1\in \fieldAlg$ be the zero function, and let $\Psi_m := \sum\limits_{k=1}^{m} \psi_k$. The local expansion of $\overline{f}$ at $\widehat{Q}$ is given by
    \begin{align*}
        \Bigg(&\sum\limits_{\omega\neq\widehat{\omega}} \sum\limits_{j,m\geq 1} \res_{\mathcal{U}}(f,Q_{\omega}\oplus m\shiftalg,j)\cdot d_j(\omega)\cdot m \Bigg) \cdot u_{\widehat{Q}}^{-1} \\
        +\Bigg(& f(\widehat{Q}) - \sum\limits_{\omega\neq\widehat{\omega}} \sum\limits_{j,m\geq 1} \res_{\mathcal{U}} (f,Q_{\omega}\oplus m\shiftalg,j)\cdot \varphi_{\omega,j}(\widehat{Q}\ominus m\shiftalg) \\
        &+ \sum\limits_{\omega\neq \widehat{\omega}} \sum\limits_{j,m\geq 1} \res_{\mathcal{U}} (f,Q_{\omega}\oplus m\shiftalg,j)\cdot d_j(\omega)\cdot \Psi_m(\widehat{Q}\ominus\shiftalg) \Bigg) \cdot u_{\widehat{Q}}^{0} \\
        &+ g,
    \end{align*}
    where $g\in\fieldAlg$ has a zero at $\widehat{Q}$.
\end{proposition}

\begin{proof}
    We find the coefficients by analyzing the contributions of the first and second reductions to the order $1$ and $0$ residues at $\widehat{Q}$. For order $1$ residues: recall the first reduction
    \[\tilde{f} = f+\sum\limits_{\omega\neq \widehat{\omega}}\sum\limits_{m\geq 1}\sum\limits_{j\geq 1} \res_{\mathcal{U}}(f,Q_{\omega}\oplus m\shiftalg,j)\cdot \Delta^{(m)}(\varphi_{\omega,j})\]
    where $\res_{\mathcal{U}}(\Delta^{(m)}(\varphi_{\omega,j}),\widehat{Q},1) = d_j(\omega)$, and the second reduction
        \[\overline{f} = \tilde{f} - \sum\limits_{m\geq 2} \sum\limits_{k=2}^{m} \left[\Delta^{(1)}(\psi_k) \sum\limits_{\omega\neq\widehat{\omega}} \sum\limits_{j\geq 1} \res_{\mathcal{U}}(f,Q_{\omega}\oplus m\shiftalg,j)\cdot d_j(\omega)\right].\]
    where $\res_{\mathcal{U}}(\Delta^{(1)}(\psi_k),\widehat{Q},1) = -1$. Because we chose an admissible algebraic pinning we have $\res_{\mathcal{U}}(f,\widehat{Q},1)=0$, therefore
    \begin{align*}
        \res_{\mathcal{U}}(\overline{f},\widehat{Q},1) = \ & \sum\limits_{\omega\neq \widehat{\omega}}\sum\limits_{m\geq 1}\sum\limits_{j\geq 1} \res_{\mathcal{U}}(f,Q_{\omega}\oplus m\shiftalg,j)\cdot d_j(\omega) \\
        & + \sum\limits_{m\geq 2} \sum\limits_{k=2}^{m} \sum\limits_{\omega\neq\widehat{\omega}} \sum\limits_{j\geq 1} \res_{\mathcal{U}}(f,Q_{\omega}\oplus m\shiftalg,j)\cdot d_j(\omega),\\
        = \ & \sum\limits_{m\geq 1} \sum\limits_{k=1}^{m} \sum\limits_{\omega\neq\widehat{\omega}} \sum\limits_{j\geq 1} \res_{\mathcal{U}}(f,Q_{\omega}\oplus m\shiftalg,j)\cdot d_j(\omega)\\
        = \ &  \sum\limits_{\omega\neq\widehat{\omega}} \sum\limits_{j,m\geq 1} \res_{\mathcal{U}}(f,Q_{\omega}\oplus m\shiftalg,j)\cdot d_j(\omega)\cdot m
    \end{align*}
    as desired.

    Because $f$ is regular at $\widehat{Q}$, the contribution of $f$ to the order $0$ residue of $\overline{f}$ at $\widehat{Q}$ is simply the evaluation $f(\widehat{Q})$. The first reduction uses terms of the form
    \[\Delta^{(m)}(\varphi_{\omega,j}) = \varphi_{\omega,j} - \shift^{-m}(\varphi_{\omega,j}).\]
    The term $\varphi_{\omega,j}$ does not contribute by definition, whereas the contribution from $-\shift^{-m}(\varphi_{\omega,j})$ is the evaluation $-\varphi_{\omega,j}(\widehat{Q}\ominus m\shiftalg)$. The second reduction uses terms of the form
    \[\Delta^{(1)}(\psi_k) = \psi_k - \shift^{-1}(\psi_k).\]
    The term $\psi_k$ does not contribute by definition, whereas the contribution from $-\shift^{-1}(\psi_k)$ is the evaluation $-\psi_k(\widehat{Q}\ominus\shiftalg)$. Consequently the order $0$ residue of $\overline{f}$ at $\widehat{Q}$ is given by
    \begin{align*}
        \res_{\mathcal{U}}(\overline{f},\widehat{Q},0) = \ &  f(\widehat{Q}) - \sum\limits_{\omega\neq \widehat{\omega}}\sum\limits_{m\geq 1}\sum\limits_{j\geq 1} \res_{\mathcal{U}}(f,Q_{\omega}\oplus m\shiftalg,j)\cdot \varphi_{\omega,j}(\widehat{Q}\ominus m\shiftalg) \\
        +& \sum\limits_{m\geq 2} \sum\limits_{k=2}^{m}  \sum\limits_{\omega\neq\widehat{\omega}} \sum\limits_{j\geq 1} \res_{\mathcal{U}}(f,Q_{\omega}\oplus m\shiftalg,j)\cdot d_j(\omega)\psi_k(\widehat{Q}\ominus\shiftalg)
    \end{align*}
    \begin{align*}
    = \ &f(\widehat{Q}) - \sum\limits_{\omega\neq \widehat{\omega}}\sum\limits_{j,m\geq 1} \res_{\mathcal{U}}(f,Q_{\omega}\oplus m\shiftalg,j)\cdot \varphi_{\omega,j}(\widehat{Q}\ominus m\shiftalg) \\
        +& \sum\limits_{m\geq 1} \sum\limits_{k=1}^{m}  \sum\limits_{\omega\neq\widehat{\omega}} \sum\limits_{j\geq 1} \res_{\mathcal{U}}(f,Q_{\omega}\oplus m\shiftalg,j)\cdot d_j(\omega)\psi_k(\widehat{Q}\ominus\shiftalg)\\
    = \ &f(\widehat{Q}) - \sum\limits_{\omega\neq \widehat{\omega}}\sum\limits_{j,m\geq 1} \res_{\mathcal{U}}(f,Q_{\omega}\oplus m\shiftalg,j)\cdot \varphi_{\omega,j}(\widehat{Q}\ominus m\shiftalg) \\
        +&   \sum\limits_{\omega\neq\widehat{\omega}} \sum\limits_{j,m\geq 1} \res_{\mathcal{U}}(f,Q_{\omega}\oplus m\shiftalg,j)\cdot d_j(\omega)\Psi_k(\widehat{Q}\ominus\shiftalg)
    \end{align*}
    as desired.
\end{proof}

This finishes the reduction process. We now have the proper definition of the panorbital residues of an elliptic function.

\begin{definition}
    Let $\Xi$ be an admissable pinning of $f\in\fieldAlg$ with associated coherent set $\mathcal{U}$ of local uniformizers. Let $\Psi_m := \sum\limits_{k=1}^{m} \psi_k$. The \defemph{algebraic panorbital residues} of order $1$ and $0$ of $f$ relative to $\Xi$ are defined to be
    \begin{align*}
        \pano_{\Xi}(f,1) = \ & \sum\limits_{\omega\neq\widehat{\omega}} \sum\limits_{j,m\geq 1} \res_{\mathcal{U}}(f,Q_{\omega}\oplus m\shiftalg,j)\cdot d_j(\omega)\cdot m,\\
        \pano_{\Xi}(f,0) = \ & f(\widehat{Q}) - \sum\limits_{\omega\neq\widehat{\omega}} \sum\limits_{j,m\geq 1} \res_{\mathcal{U}} (f,Q_{\omega}\oplus m\shiftalg,j)\cdot \varphi_{\omega,j}(\widehat{Q}\ominus m\shiftalg) \\
        &+ \sum\limits_{\omega\neq \widehat{\omega}} \sum\limits_{j,m\geq 1} \res_{\mathcal{U}} (f,Q_{\omega}\oplus m\shiftalg,j)\cdot d_j(\omega)\cdot \Psi_m(\widehat{Q}\ominus\shiftalg).
    \end{align*}
\end{definition}

\subsection{Algebraic Statement of the Main Result}

With this definition of algebraic panorbital residues, the results of Proposition~\ref{prop:algReductionEnd}, and the additional result \cite[Proposition~B.8]{Dreyfus2018}, we can now conclude the following main theorem.
\newpage 
\begin{theorem}
    Let $f\in\fieldAlg$ be elliptic and let $\Xi$ be an admissable algebraic pinning. Then $f$ is summable if and only if the algebraic orbital residues of $f$ and the algebraic panorbital residues of $f$ relative to $\Xi$ all vanish.
\end{theorem}

\chapter{Future Directions: Differential Dependence}
\label{c:future}

As alluded to in Chapters~\ref{c:intro} and \ref{c:basicResidues}, criteria for summability play a basic role in calculating the differential Galois groups of linear difference equations. While this has not been explicitly shown in the elliptic case, it is the case that summability is an integral facet of the current algorithms, developed in \cite{Arreche2016} and \cite{Arreche2022}, for these calculations in the shift and $q$-difference cases discussed in Section~\ref{s:residueExamples}. In particular there exist results that concern differential dependence of a particular family of solutions $\{z_k\}$, $k=1,\dots,m$, to a system of first-order linear difference equations $\tau(z_k)=a_kz_k$. With criteria now established for summability in the elliptic case, a similar result should hold in principle.

\section{Previous Results on Differential Dependence and Summability}

Let us first review the requisite concepts from differential algebra. Consider a \defemph{$\delta$-algebra} $R$: an algebra $R$ equipped with a derivation $\delta$. A \defemph{differential polynomial} $L(Y_1,\dots,Y_n)$ is an element of the polynomial ring over $R$ with infinitely many variables $\delta^{(i)}(Y_j)$ for $i\geq 1$ and $j=1,\dots,n$. If furthermore all the terms of $L(Y_1,\dots,Y_n)$ belong to the $R$-span of the variables $\delta^{(i)}(Y_j)$, we call the differential polynomial a \defemph{homogeneous linear differential polynomial}. If $S$ is a $\delta$-algebra containing $R$, then $z_1$, $\dots$, $z_n$ in $S$ are \defemph{differentially dependent} over $R$ if there exists a nonzero differential polynomial $L(Y_1,\dots,Y_n)$ with coefficients in $R$ such that $L(z_1,\dots,z_n)=0$.

We have the following lemma on differential dependence of variables. This result is also proved in \cite{kolchinDifferentialAlgebraAlgebraic1973}, but the statement that we need is elementary so we give the proof here.

\begin{lemma}\label{lem:difDep}
    Let $R$ be a $\delta$-field, and $S\supset R$ be a $\delta$-field containing $z_1$, $\dots$, $z_n$. Then $z_1$, $\dots$, $z_n$ are differentially dependent over $R$ if and only if the logarithmic derivatives
    \[\frac{\delta(z_1)}{z_1},\quad\dots,\quad \frac{\delta(z_n)}{z_n}\]
    are differentially dependent over $R$.
\end{lemma}

\begin{proof}
    If $\frac{\delta(z_1)}{z_1}$ through $\frac{\delta(z_n)}{z_n}$ are differentially dependent over $R$ then so are $z_1$ through $z_n$, because any relation among $\frac{\delta(z_k)}{z_k}$ are relations among $z_k$. We show the other direction. We have the following diagram of inclusions of $\delta$-algebras:
    \[\begin{tikzcd}
	{R\left<z_1,\dots,z_n\right>} & \supset & {R\left<z_1,\dots,z_{n-1}\right>} \\
	\cup && \cup \\
	{R\left< \frac{\delta(z_1)}{z_1},\dots,\frac{\delta(z_n)}{z_n}\right>} & \supset & {R\left< \frac{\delta(z_1)}{z_1},\dots,\frac{\delta(z_{n-1})}{z_{n-1}}\right>}
    \end{tikzcd}\]
The extension $R\left<z_1,\dots,z_{n-1}\right>\Big/R\left<\frac{\delta(z_1)}{z_1},\dots,\frac{\delta(z_{n-1})}{z_{n-1}}\right>$ has finite \emph{algebraic} transcendence degree, in fact at most $n$, in view of
\[R\left<\frac{\delta(z_1)}{z_1},\dots,\frac{\delta(z_{n-1})}{z_{n-1}}\right>(z_1,\dots,z_n) = R\left<z_1,\dots,z_n\right>.\]
Because $z_1$ through $z_n$ are differentially dependent there exists a suitable relabeling of the variables such that $R\left<z_1,\dots,z_n\right> / R\left<z_1,\dots,z_{n-1}\right>$ has finite algebraic transcendence degree: given a differential relationship between $z_1$ through $z_n$ from which $z_n$ or its derivatives cannot be eliminated, adjoining sufficiently many terms of the form $\delta^{(j)}(z_n)$ to $R\left<z_1,\dots,z_{n-1}\right>$ for $j\in\mathbb{Z}_{\geq 0}$ will produce an algebraic relationship between some $\delta^{(j)}(z_n)$, the lower-order derivatives of $z_n$, and the elements of $R\left<z_1,\dots,z_{n-1}\right>$. Differentiating that particular relationship shows that $\delta^{(\ell)}(z_n)$ are already expressed algebraically in terms of lower-order derivatives of $z_n$ for $\ell>j$, implying that the extension $R\left<z_1,\dots,z_n\right> / R\left<z_1,\dots,z_{n-1}\right>$ has finite algebraic transcendence degree.

This shows that the extension $R\left<z_1,\dots,z_n\right>\Big/R\left<\frac{\delta(z_1)}{z_1},\dots,\frac{\delta(z_{n-1})}{z_{n-1}}\right>$ has finite algebraic transcendence degree, due to algebraic transcendence degree being additive in towers. Therefore the intermediate extension \[R\left<\frac{\delta(z_1)}{z_1},\dots,\frac{\delta(z_{n})}{z_{n}}\right>\bigg/R\left<\frac{\delta(z_1)}{z_1},\dots,\frac{\delta(z_{n-1})}{z_{n-1}}\right>\]
has finite algebraic transcendence degree. This implies that $\frac{\delta(z_1)}{z_1}$, $\dots$, $\frac{\delta(z_{n})}{z_{n}}$ are differentially dependent over $R$: if they were not, then by adjoining all the derivatives of $\frac{\delta(z_{n})}{z_{n}}$ to $R\left<\frac{\delta(z_1)}{z_1},\dots,\frac{\delta(z_{n-1})}{z_{n-1}}\right>$ one would obtain an extension of \emph{infinite} algebraic transcendence degree. This proves the result.
\end{proof}

The pioneering work of \cite{Hardouin2008} establishing the differential Galois theory of linear difference equations includes a result that links differential dependence to summability. For the formulation of this result, define a \defemph{$\shift\delta$-algebra} to be an algebra, either a field of a ring, equipped with an automorphism $\shift$ and a derivation $\delta$ satisfying $\shift\circ\delta = \delta\circ\shift$. The following result is \cite[Proposition~3.1]{Hardouin2008}: 

\begin{proposition}
    Let $K$ be a $\shift\delta$-field with the subfield of $\shift$ constants $K^{\shift}$ differentially closed and let $S\supset K$ be a $\shift\delta$-ring such that the subring of $\shift$ constants $S^{\shift}$ is precisely $K^{\shift}$. Let $b_1,\dots,b_n\in K$ and $z_1\dots z_n\in S$ satisfy
    \begin{equation}
        \shift(z_i)-z_i = b_i,\qquad i=1,\dots,n.\label{eq:hardouinSystem}
    \end{equation}
    Then $z_1,\dots,z_n$ are differentially dependent over $K$ if and only if there exists a nonzero homogeneous linear differential polynomial $L(Y_1,\dots,Y_n)$ with coefficients in $K^{\shift}$, and an element $f\in K$, such that
    \[L(b_1,\dots,b_n)=\shift(f)-f.\]
\end{proposition}

Deeper than just a connection between differential dependence and summability, this result gives a test for when there exist differerential symmetries amongst a set of solutions $z_1$, $\dots$, $z_n$ to the nonhomogeneous system~\eqref{eq:hardouinSystem} of linear difference equations. If there does not exist $L(Y_1,\dots,Y_n)$ as described above, then no differential symmetries exist between $z_1$, $\dots$, $z_n$. If there does exist $L(Y_1,\dots,Y_n)$ as above, there \emph{do} exist differential symmetries between $z_1$, $\dots$, $z_n$, and searching for them is worthwhile. In the broader Galois theoretic context, this result gives us a gauge of the size of the differential Galois group associated with system~\eqref{eq:hardouinSystem}: the existence of differential symmetries implies a smaller Galois group, and the lack of existence implies the Galois group is as large as possible.

It would certainly be daunting to search for all homogeneous linear differential polynomials $L$ and elements $f\in K$ satisfying $L(b_1,\dots,b_n)=\shift(f)-f$. The prevailing strategy for computing differential Galois groups is to show that in particular cases it suffices to search for differential polynomials $L$ of particular forms. The rational shift context of \cite[Proposition~2.1]{Arreche2016} and the rational $q$-difference context of \cite[Corollary~3.5]{Arreche2022} concern systems of \emph{homogeneous} linear difference equations
\begin{equation}
\shift(z_i) = a_iz_i.\label{eq:homSystem}
\end{equation}
Taking the logarithmic derivative shows that this system gives rise to a nonhomogeneous system as in \eqref{eq:hardouinSystem}:
\begin{align}
\frac{\delta(\shift(z_i))}{\shift(z_i)} &= \frac{\delta(a_iz_i)}{a_iz_i}\nonumber\\
&= \frac{\delta(a_i)}{a_i} + \frac{\delta(z_i)}{z_i}\nonumber\\
\Longrightarrow\quad  \shift\left(\frac{\delta(z_i)}{z_i}\right) - \frac{\delta(z_i)}{z_i} &= \frac{\delta(a_i)}{a_i},\label{eq:transformEquation}
\end{align}
using the fact that $\shift$ and $\delta$ commute. The authors of \cite{Arreche2016} and \cite{Arreche2022} apply \cite[Proposition~3.1]{Hardouin2008} to show that for determining the differential dependence of $z_1$, $\dots$, $z_n$, it suffices to search for differential polynomials $L$ of the form
\[L(Y_1,\dots,Y_n) = \ell_1Y_1 + \cdots + \ell_n Y_n,\]
where $\ell_1,\dots,\ell_n\in\mathbb{Z}$. These are in fact polynomials in the \emph{algebraic} sense. Linear polynomials with integer coefficients form a much smaller search space and create much more manageable algorithms for determining differential dependence of solutions to system~\eqref{eq:homSystem}.

Here we give the precise wording of the result \cite[Corollary~3.5]{Arreche2022}. While it rests on an application of \cite[Proposition~3.1]{Hardouin2008} it twists the condition of summability in a manner that is mirrored in the elliptic case. For this proposition, $K=C(x)$ where $C$ is a differentially closed field with derivation $\delta$ of characteristic $0$, and $\shift$ is the automorphism on $C(x)$ defined by $x\mapsto qx$ where $q\in C^{\delta}$ is a nonzero $\delta$ constant that is not a root of unity.

\begin{proposition}
    Let $K$ be a $\shift\delta$-field, and let $R$ be a $\shift\delta$-algebra containing $K$ such that the subalgebras of $\shift$ constants $K^{\shift}$ and $R^{\shift}$ are the same field $C$. Let $a_1,\dots,a_n\in C^{\delta}(x)$ be rational functions with coefficients in the $\delta$-constants of $C$. Let $z_1,\dots,z_n\in R^{\times}$ satisfy
    \[\shift(z_i) = a_iz_i,\qquad i=1,\dots,n.\]
    Then $z_1$, $\dots$, $z_n$ are differentially dependent over $K$ if and only if there exist $\ell_1,\dots,\ell_n\in\mathbb{Z}$ not all zero and with $\gcd(\ell_1,\dots,\ell_n)=1$; $c\in \mathbb{Z}$; and $f\in K$ such that
    \[\ell_1 \frac{\delta(a_1)}{a_1} + \cdots + \ell_n\frac{\delta(a_n)}{a_n} = \shift(f)-f+c.\]
\end{proposition} \medskip \medskip

This result is shown in \cite{Arreche2022} by analyzing the $q$-discrete residues of $\frac{\delta(a_i)}{a_i}$. It stands to reason that similar results could be shown in the elliptic case with orbital and panorbital residues.

\section{Differential Dependence in the Elliptic Case}

With this context we are equipped to start analyzing the connection of summability to differential dependence in the elliptic case. There are results that we can formulate immediately using the technology of orbital residues, but it is currently unknown if there exists a result as strong as \cite[Corollary~3.5]{Arreche2022} in the elliptic case.

Let $\fieldC$ and $\shift$ be as in Chapter~\ref{c:analytic}, and let $\delta$ be the derivation $\frac{d}{dz}$ on $\mathbb{C}$. Note that $\shift$ and $\delta$ commute due to the chain rule, and this places a $\shift\delta$-field structure on $\fieldC$. We will need the following computational lemma:

\begin{lemma}
\label{lem:orderReduction}
        Suppose $a\in\fieldC^{\times}$, $\omega\in\allOmegas$, and $r\in \mathbb{Z}_{\geq 0}$. Then
    \[\ores\left( \delta^r\left(\frac{\delta(a)}{a}\right),\omega,r+1\right) = (-1)^r\cdot r!\cdot \ores \left( \frac{\delta(a)}{a},\omega,1\right)\in\mathbb{Z}.\]
\end{lemma}

Recall that the orbital residues are computed with respect to an analytic pinning: the pinning is still assumed to be fixed a priori.

\begin{proof}
    Recall from the discussion of Section~\ref{s:sigmaExample} that the logarithmic derivative $\frac{\delta(a)}{a}$ has a $\zeta$-expression of the form
    \[\frac{\delta(a)}{a} = b+\sum\limits_{k=1}^{n} m_k\zeta(z-\alpha_k),\]
    where $b\in\mathbb{C}$, $m_k\in\mathbb{Z}$, and $\{\alpha_1,\dots,\alpha_n\}$ are representatives of a complete set of zeros and poles of $a$ modulo $\Lambda$. It follows that
    \[\ores\left(\frac{\delta(a)}{a},\omega,1\right)\in\mathbb{Z}\ \text{ for all }\omega\in\allOmegas.\]
    For $r\geq 1$ the $r$th derivative of $\frac{\delta(a)}{a}$ is
    \[\delta^r\left(\frac{\delta(a)}{a}\right) = \sum\limits_{k=1}^{n} m_k\zeta^{(r)}(z-\alpha_k),\]
    implying that the analytic orbital residue of $\delta^r(\delta(a)/a)$ of order $r+1$ is
    \[\ores\left( \delta^r\left(\frac{\delta(a)}{a}\right),\omega,r+1\right) = \sum\limits_{k=1}^{n} m_k\cdot r!\cdot (-1)^r = (-1)^r\cdot r!\cdot \ores\left(\frac{\delta(a)}{a},\omega,1\right).\]
    These numbers are integers for all $\omega\in\allOmegas$, which proves the result.
\end{proof}

The following result can be shown now in the analytic setting, and it only relies on the technology of orbital residues developed in \cite{Dreyfus2018} and reformulated analytically in Chapter~\ref{c:analytic}:

\begin{proposition}
\label{prop:onlyUsesOrbitalResidues}
    Let $R$ be a $\shift\delta$-$\fieldC$-algebra such that the elements of $R$ invariant under $\shift$ are precisely the elements of $\fieldC$ invariant under $\shift$: denote these subalgebras $R^{\shift}$ and $\fieldC^{\shift}$ respectively. Suppose $a_1,\dots,a_n\in \fieldC^{\times}$ and $y_1,\dots,y_n\in R^{\times}$ satisfy
    \[\shift(y_i) = a_iy_i;\quad i=1,\dots,n.\]
    Then $y_1$, $\dots$, $y_n$ are differentially dependent over $\fieldC$ if and only if there exists $\ell_1,\dots,\ell_n\in\mathbb{Z}$ not all zero and an element $f\in \fieldC$ such that
    \[\ell_1\delta\left(\dfrac{\delta(a_1)}{a_1}\right)+\cdots + \ell_n\delta\left(\dfrac{\delta(a_n)}{a_n}\right)=\shift(f)-f.\]
\end{proposition} \medskip\medskip

The proof of this result shall mimic that of \cite[Corollary~3.5]{Arreche2022}.

\begin{proof}
    For the if direction, assume that for some $f\in\fieldC$ and some integers $\ell_1$, $\dots$, $\ell_n$ not all zero,
    \[\sum\limits_{i=1}^{n} \ell_i\delta\left(\frac{\delta(a_i)}{a_i}\right)=\shift(f)-f\]
    holds. Because $\delta$ and $\shift$ commute,
    \[\frac{\delta(a_i)}{a_i} = \shift\left(\frac{\delta(y_i)}{y_i}\right) - \frac{\delta(y_i)}{y_i}\]
    for all $1\leq i\leq n$. Therefore
    \begin{align*}
        \shift\left(\sum\limits_{i=1}^{n}  \ell_i\delta\left(\frac{\delta(y_i)}{y_i}\right)\right) - \sum\limits_{y=1}^{n} \ell_i\delta\left(\frac{\delta(y_i)}{y_i}\right) &= \shift(f)-f\\
        \Longrightarrow\quad \shift\left(-f+\sum\limits_{i=1}^{n} \ell_i\delta\left(\frac{\delta(y_i)}{y_i}\right)\right) &= -f+\sum\limits_{i=1}^{n} \ell_i\delta\left(\frac{\delta(y_i)}{y_i}\right).
    \end{align*}
    It follows that
    \[-f+\sum\limits_{i=1}^{n} \ell_i\delta\left(\frac{\delta(y_i)}{y_i}\right)\in R^{\shift}.\]
    From the assumptions of the proposition this implies that the above expression is also in $\fieldC^{\shift}$, which is precisely the field $\mathbb{C}$ because the shift was chosen to be a non-torsion point with respect to $\Lambda$. Therefore $y_1$, $\dots$, $y_n$ are indeed differentially dependent over $\fieldC$.
    
    For the only if direction, assume that $y_1$, $\dots$, $y_m$ are differentially dependent over $\fieldC$. Then from Lemma~\ref{lem:difDep} we also have that $\frac{\delta(y_1)}{y_1}$, $\dots$, $\frac{\delta(y_n)}{y_n}$ are differentially dependent over $\fieldC$. From the derivation of equation~\eqref{eq:transformEquation}, they satisfy the nonhomogeneous linear difference equations
    \[\shift\left(\frac{\delta(y_i)}{y_i}\right) - \frac{\delta(y_i)}{y_i} = \frac{\delta(a_i)}{a_i},\qquad 1\leq i\leq n.\]
    Then \cite[Proposition~3.1]{Hardouin2008} implies that there exists a nonzero homogeneous linear differential polynomial $L$ and an elliptic $f_1\in\fieldC$ such that
    \[g:=L\left(\frac{\delta(a_1)}{a_1},\dots,\frac{\delta(a_n)}{a_n}\right)=\shift(f_1)-f_1.\]
    We can express $g$ on the level of individual terms as
    \begin{equation}
    g = \sum\limits_{i=1}^{n}\sum\limits_{j=0}^{r_i} c_{i,j}\delta^j\left(\frac{\delta(a_i)}{a_i}\right)\label{eq:gDef}\end{equation}
    
    for $c_{i,j}\in\mathbb{C}$. Each elliptic function $\delta(a_i)/a_i\in\fieldC$ only has poles of order $1$, so $\delta^j(\delta(a_i)/a_i)$ has poles of order at most $j+1$. Therefore
    \[\ores\left(\delta^j\left(\frac{\delta(a_i)}{a_i}\right),\omega,j'\right)=0\quad\text{ for all }j'\geq j+1\text{ and } \omega\in\allOmegas.\]
    
    Let $r$ be the maximal $r_i$ that appears in equation~\eqref{eq:gDef}, i.e. let $r$ be the highest-order term of $L$. Then the calculation of the orbital residue of $g$ reduces to
    \begin{align*}
    \ores(g,\omega,r+1)&=\sum\limits_{i=1}^{n}\sum\limits_{j=1}^{r_i} c_{i,j}\cdot\ores\left(\delta^j\left(\frac{\delta(a_i)}{a_i}\right),\omega,r+1\right) \\
    &= \sum\limits_{i=1}^{n} c_{i,r}\cdot\ores\left(\delta^r\left(\frac{\delta(a_i)}{a_i}\right),\omega,r+1\right)
    \end{align*}

    From Lemma~\ref{lem:orderReduction} this simplifies to
    \[\ores(g,\omega,r+1)=(-1)^r\cdot r!\cdot \sum\limits_{i=1}^{n} c_{i,r}\cdot\ores\left(\frac{\delta(a_i)}{a_i},\omega,1\right).\]

    Because $g$ is summable, Lemma~\ref{lem:analyticOrbitalResidues} implies that all the orbital residues of $g$ vanish, so the above sum is zero for all $\omega\in\allOmegas$. That is, the $n$ infinite-dimensional vectors
    \[v_i=\underline{\ores}\left(\frac{\delta(a_i)}{a_i},\omega,1\right),\qquad 1\leq i\leq n\]
    with entries corresponding to $\omega\in\allOmegas$ are $\mathbb{C}$-linearly dependent. Recall from Section~\ref{s:sigmaExample} that the orbital residues of the $\frac{\delta(a_i)}{a_i}$ are all integers, so these vectors are in fact $\mathbb{Z}$-linearly dependent. Consequently there exist integers $\ell_1$, $\dots$, $\ell_m$ not all zero such that all the orbital residues of
    \[\ell_1 \cdot\frac{\delta(a_1)}{a_1} + \cdots + \ell_n\cdot\frac{\delta(a_n)}{a_n}\]
    vanish. From Corollary~\ref{cor:analyticDHRS} that there must there must exist some other $f_2\in\fieldC$, and constants $c_1,c_2\in\mathbb{C}$, such that
    \[\sum\limits_{i=1}^{n}\ell_i\frac{\delta(a_i)}{a_i}=\shift(f_2)-f_2 + c_1 + c_2(\zeta(z+\shiftnumber) - \zeta(z)).\]
    From Lemma~\ref{lem:analyticLPSummable} it follows that
    \[\sum\limits_{i=1}^{n}\ell_i\delta\left(\frac{\delta(a_i)}{a_i}\right)=\shift(\delta(f_2)-c_2\wp(z))-\big(\delta(f_2)-c_2\wp(z)\big).\]
    Let $f=\delta(f_2)-c_2\wp(z)$. This proves the result.
\end{proof}

\begin{remark}
    This proposition is phrased as differential dependence of $y_1$, $\dots$, $y_n$ being equivalent to the existence of parameters $f\in\fieldC$ and $\ell_1,\dots,\ell_n\in\mathbb{Z}$ satisfying properties. However the proof implies that there is a much easier test for differential dependence: find $\ell_1, \dots, \ell_n\in\mathbb{Z}$ such that the orbital residues of
    \[\ell_1 \cdot\frac{\delta(a_1)}{a_1} + \cdots + \ell_n\cdot\frac{\delta(a_n)}{a_n}\]
    all vanish. This removes the task of finding $f$ to determine differential dependence of $y_1$ through $y_n$. This elimination of quantifiers is a facet of the philosophy behind formulating summability in terms of residues.
\end{remark} \medskip \medskip

Notice that the proof of Proposition~\ref{prop:onlyUsesOrbitalResidues} only uses the technology of orbital residues. That of panorbital residues does not appear. As a result, the expressions $\ell_i \cdot\frac{\delta(a_i)}{a_i}$ get wrapped in an additional derivation $\delta$. For the purposes of computing the differential Galois group of a system of linear difference equations, this obscures structural information about the group. However, the panorbital residues are a finer-tuned construction. This introduces the question:

\begin{question}
    Can we use the panorbital residues to alter the previous result in order to glean more information about differential Galois groups?    
\end{question}

An example of this concept in action is the proof of \cite[Corollary~3.5]{Arreche2022}, which uses a ``residue at infinity'' along with the $q$-difference analog of orbital residues to get a more precise statement.

\begin{thesisbib}

\end{thesisbib}

\end{document}